\documentclass[a4paper,10pt,twoside]{article}
\usepackage[german,french,english]{babel}
\usepackage[T1]{fontenc}
\usepackage[utf8]{inputenc}
\usepackage[numbers,sort&compress]{natbib} 
\usepackage{amssymb}
\usepackage{enumerate}
\usepackage{amsmath, amsthm}
\usepackage{graphicx}
\usepackage{nicematrix}
\usepackage{lmodern}
\usepackage[french]{babel}
\usepackage{enumitem}
\usepackage{hyperref}
\usepackage[a4paper,margin=2.5cm]{geometry}

\usepackage[most]{tcolorbox}
\usepackage{fancyhdr}
\usepackage{float}
\usepackage{subcaption}

\usepackage{bbm}
\usepackage{dsfont}

\newcommand{\dd}{\,\mathrm{d}}
\newcommand{\R}{\mathbb{R}}
\newcommand{\N}{\mathbb{N}}
\newcommand{\E}{\mathbb{E}}
\renewcommand{\P}{\mathbb{P}}

\newtheorem{Theorem}{Theorem}[section]
\newtheorem{Definition}[Theorem]{Definition}
\newtheorem{Proposition}[Theorem]{Proposition}

\newtheorem{Lemma}[Theorem]{Lemma}

\newtheorem{Remark}[Theorem]{Remark}

\newtheorem{assumption}{Assumption}[section]
\newtheorem*{note}{Note}

\newcommand{\p}{\mathbb{P}}
\newcommand{\F}{\mathbb{F}}

\newcommand{\cF}{{\mathcal F}}
\newcommand{\cA}{{\mathcal A}}

\newcommand{\id}{\mathds{1}}
\newcommand{\as}{\mbox{{\rm a.s.}}}

\usepackage{authblk}
	
	\newcommand{\T}{\top}
	\renewcommand{\c}{\alpha}
	\newcommand{\Mid}{{\ \Big|\ }}

	\DeclareMathOperator{\diag}{diag}
    
\title{Optimal Merton's Problem under Multivariate Affine Volterra Models with Jumps\footnote{ The work of the first author is partially supported by
  Agence Nationale de la Recherche (ReLISCoP grant ANR-21-CE40-0001).  The second author gratefully acknowledge support from \textit{Ecole Doctorale Sciences Mathematiques de Paris Centre}.}}
\author[1,2]{Sigui Brice Dro\footnote{E-mail: {\tt dro@lpsm.paris} }}
\author[1]{Emmanuel Gnabeyeu\footnote{E-mail: {\tt emmanuel.gnabeyeu\_mbiada@sorbonne-universite.fr} }}

\affil[1]{Laboratoire de Probabilités, Statistique et Modélisation, UMR 8001, Sorbonne Université and Universit\'e Paris Cit\'e, 4 pl. Jussieu, F-75252 Paris Cedex 5, France.}
\affil[2]{BPCE SA, 7 promenade Germaine Sablon Paris 75013, France.}

\begin{document}

\maketitle

\begin{abstract}
This paper is concerned with portfolio selection for an investor with exponential, power, and logarithmic utility in multi-asset financial markets allowing jumps. We investigate the classical Merton's portfolio optimization problem in a Volterra stochastic environment described by a multivariate Volterra--Heston model with jumps driven by an independent Poisson random measure. 
	Owing to the non-Markovian and non-semimartingale nature of the model, classical stochastic control techniques are not directly applicable. Instead, the problem is tackled using the martingale optimality principle by constructing a family of supermartingale processes characterized via solutions to an original Riccati backward stochastic differential equation with jumps (Riccati BSDEJ).
	The resulting optimal strategies for Merton's problems, as well as the corresponding indifference prices, are derived in semi-closed form depending on the solutions to time-dependent multivariate Riccati-Volterra integral equations with L\'evy exponential jump compensator, while the optimal value is expressed using the solution to this original Riccati BSDEJ.
    Numerical experiments on a two-dimensional rough Heston model illustrate the impact of both path roughness and jumps components on the value function and optimal strategies  in the Merton problem.
     
\end{abstract}


\tableofcontents


\section{Introduction}
The maximization of a portfolio's 
expected utility of terminal wealth with respect to a given utility function has served as a cornerstone of mathematical finance and has attracted considerable attention from both practitioners and academicians. It dates back to Merton~\cite{Merton1969,Merton1971}, who studied the joint problem of investment and consumption choice and provided the classical financial economic framework for understanding how market volatility affects investment demands.
Such a problem becomes particularly relevant in the presence of trading constraints, as these generally destroy market completeness. In this context, the valuation of contingent claims is no longer unique, and utility-based indifference pricing provides a natural alternative.

\medskip
\noindent
Such utility maximization problem has been extensively studied in various continuous-time Brownian-motion-based models under general constraints. Several approaches have been developed, including convex duality methods, stochastic control techniques based on Hamilton--Jacobi--Bellman (HJB) equations, and the martingale optimality principle via backward stochastic differential equations (BSDEs) (see, e.g., ~\cite{CvitanicKaratzas1992,HuImkellerMueller2005}).


\medskip
\noindent
The literature has also addressed mixed Brownian--Poisson frameworks. In particular, one may distinguish between settings where the contingent claim depends on both a Brownian motion and an independent jump Poisson process, while the underlying asset dynamics remain continuous~\cite{Becherer2006}, and settings where the asset dynamics themselves are driven by both a Brownian motion and a Poisson jump random measure~\cite{Morlais2008,Morlais2009,Bank2022Merton, turki2026portfolio}.

\medskip
\noindent Over the past
decade, the modeling of asset price dynamics has undergone a paradigm shift with the empirical observation that both implied and realized volatilities of major financial indices exhibit significantly rougher sample paths~\citep{GatheralJR2018} than those generated by classical Brownian-motion-based models.
This observation has sparked a rapidly expanding body of research on rough volatility models~\cite{el2019characteristic, ElEuchFukasawaRosenbaum2018}. 
 Recent remarkable advances include the introduction of the so-called \textit{fake stationary Volterra Heston models} in~\cite{EGnabeyeuPR2025, EGnabeyeuR2025}.

 \medskip
\noindent
 This broader class of volatility models, is obtained by modeling the volatility process as an \textit{affine stochastic Volterra integral equation} (SVIE for short) of \textit{convolution type} with possibly time-dependent coefficients. These processes overcome certain modeling shortcomings of classical affine processes (such as those used in the celebrated Heston model~\cite{Heston1993}), as they allow for trajectories with regularity differing from that of the paths of Brownian motion. In particular, singular kernels give rise to rough processes in the spirit of~\cite{el2019characteristic, ElEuchFukasawaRosenbaum2018}. 

\medskip
\noindent Merton’s portfolio optimization problem has been investigated in the Volterra framework. However, the non-Markovian nature of the volatility process prevents the direct application of classical stochastic control methods based on the Hamilton--Jacobi--Bellman (HJB) partial differential  equation.
In order to circumvent this difficulty,~\cite{HanWong2020b,AichingerDesmettre2021, Gnabeyeu2026b} inspired by~\cite{fouque2018aoptimal} apply the martingale optimality principle in the highly degenerate correlation settings, by adopting a martingale distortion ansatz, originally introduced in the seminal paper by~\cite{zariphopoulou2001solution} and later transferred to a non-Markovian setting with H\"older-type inequalities in \cite{Teh04}. In the non-degenerate multidimensional correlation setting, the martingale distortion approach is no longer applicable. To overcome this issue, \cite{AichingerDesmettre2021} 
provide a verification argument using calculus of convolutions and resolvents in the power utility case, while \cite{Gnabeyeu2026b} establish the verification through a BSDE approach.
\medskip
\noindent
In particular, in~\cite{Gnabeyeu2026b}, the author use a martingale optimality principle combined with a BSDE-based verification argument
to characterize
the optimal strategy and the optimal value function. This martingale optimality principle allows to derive a Riccati 
Backward Stochastic Differential Equation (Riccati BSDE for short) that has a quadratic growth in the control term denoted $\Lambda$. 
The author show that solvability of this non-linear Riccati BSDE is equivalent to solvability in the set of continuous functions, of a system of inhomogeneous Ricatti-Volterra equation and hence get an optimal solution for both the strategy and the value function.

\medskip
\noindent
 The goal of this paper is to extend the results of~\cite{HanWong2020b,AichingerDesmettre2021,Gnabeyeu2026b} by incorporating jumps into the affine stochastic Volterra integral equations describing the volatility process.
Our study is motivated by financial models of stock volatility. In particular, the empirical findings of~\cite{TodorovTauchen2011} suggest that an adequate and complete description of volatility should account for both path roughness and discontinuous movements induced by jumps. This observation also motivates the development of volatility derivative models that simultaneously incorporate rough volatility and volatility jumps; see also~\cite{Xia2022} for further interesting discussion on the topic. 
 
 \medskip
\noindent
The literature on stochastic Volterra equations with jumps is instead very recent, and shows a growing interest. \cite{AlfonsiSzulda2025} establish conditions upon the kernel  and coefficients for the strong existence and pathwise uniqueness of a non-negative c\`adl\`ag solution to stochastic volterra integral equation with jumps. \cite{BondiLivieriPulido2024} derived, under
an affine structure imposed upon the coefficients, a semi-explicit formula for the Fourier–Laplace
transform of the solution.

\medskip
\noindent
We propose to study in this work, utility maximization problem in a framework where volatility incorporates memory
effects and both Brownian and Poisson sources of randomness, as considered in~\cite{BondiLivieriPulido2024,AlfonsiSzulda2025}, while the underlying asset prices remain continuous, in the spirit of~\cite{HanWong2020b,AichingerDesmettre2021,Gnabeyeu2026b}.


\medskip
\noindent {\bf Main contributions.}
Building upon recent developments in volatility~\cite{GnabeyeuKarkarIdboufous2024} and
Volterra models ~\cite{TomasRosenbaum2021,
	EGnabeyeuPR2025,EGnabeyeuR2025} and motivated by recent works and advances on utility maximization under multivariate Volterra volatility modeling \cite{HanWong2020b, AichingerDesmettre2021, Gnabeyeu2026b}, this paper contributes to the literature along two main directions:
\begin{itemize}
	\item[(i)] We introduce a class of \textit{ multivariate affine Volterra stochastic volatility models with jumps } that capture key stylized features of financial markets, including heterogeneous roughness across assets and leverage effects namely, dependence between asset returns and their respective volatilities while incorporating jumps of volatilities.
	\item[(ii)] We establish that this framework remains analytically tractable, allowing for an explicit solution to the Merton problem despite the non-Markovian and multivariate nature of the dynamics, with the optimal value characterized by the solution of an associated Riccati BSDE with jumps (Riccati BSDEJ) whose generator is quadratic in the Brownian stochastic integrand (denoted by \(\Lambda\)) and with an exponential term involving the Poisson stochastic integrand variable (denoted by \(U\)). 
    \item[(iii)] It further extends the models of~\cite{TomasRosenbaum2021, Gnabeyeu2026b} to an inhomogeneous setting, following~\cite{EGnabeyeuPR2025,EGnabeyeuR2025}, while additionally incorporating jumps driven by a Poisson random measure. Our results generalize those of~\cite{HanWong2020b, AichingerDesmettre2021} for rough Heston models, as well as~\cite{Gnabeyeu2026b} in the context of the so-called ``multivariate fake stationary rough Heston model''.
\end{itemize}

\medskip
\noindent {\sc \textbf{Organization of the Work.}}
\noindent 
The remainder of the paper is organized as follows. Section~\ref{sect-preliminaries} provides an overview of the probabilistic model of market that will be used throughout the paper. In particular, we introduce a multi-asset financial market in which volatility is modeled by a multivariate class of \textit{Volterra square-root processes with jumps} driven by an independent Poisson measure.
\noindent For such a market model, we investigate in section~\ref{sect-MertProblem} the classical problem of maximizing the expected utility of terminal wealth, 
for each of the exponential, power and logarithmic utility preferences, and highlight the related martingale optimality principle.
  We then provide a solution for the Merton portfolio optimization problem for a more general correlation structure between asset returns and their respective volatilities. The analysis relies on tools from backward stochastic differential equations with jumps (BSDEJs), whose solutions are explicitly characterized in terms of the solution to associated Riccati--Volterra equations by a verification method.
  In Section~\ref{Sec:Num}, we demonstrate the practical implications of our findings through numerical experiments based on a two-dimensional rough Heston model with jumps and depict the impact of jumps on the variance process. Finally, Section~\ref{sect:proofMresult} is devoted to the proofs of the main results. 
  
\medskip
\noindent {\sc \textbf{Notations.}} 

\smallskip
\noindent $\bullet$ Denote $\mathbb{T} = [0, T] \subset \mathbb{R}_+$, ${\rm Leb}_d$ the Lebesgue measure on $(\R^d, {\cal B}or(\R^d))$, $d\geq1$.

\noindent $\bullet$ $\mathbb{X} := {\cal C}([0,T], \R^d)  (\text{resp.} \quad {\mathcal C_0}([0,T], \R^d))$ denotes the set of continuous functions (resp. null at 0)  from $[0,T]$ to $\R^d $ and ${\cal B}or(\mathbb{X})$ denotes the  Borel $\sigma$-field of $\mathbb{X}$ induced by the $\sup$-norm topology. 

\smallskip
\noindent $\bullet$ Let $E$ represents a subspace of $\mathbb{R}$ and  $\mathcal{N}(0,1)(de)$ denote the density of the standard gaussian law.

\smallskip 
\noindent $\bullet$ For $p\in(0,+\infty)$, $L_{\mathbb H}^p(\P)$ or simply $L^p(\P)$ denote the set of  $\mathbb H$-valued random vectors $X$  defined on a probability space $(\Omega, {\cal A}, \P)$ such that $\|X\|_p:=(\E[\|X\|_{\mathbb H}^p])^{1/p}<+\infty$. 

\smallskip 
\noindent $\bullet$ Let \(\mathcal{M}\) denote the space of all $(\R_+, {\cal B}or(\R))$-measurable functions \(m\) on \(\mathbb{R}_+\) such that the restriction \(m|_{[0, T]}\), for any \(T > 0\), is a \(\mathbb{R}\)-valued finite measure (i.e. the restriction $m|_{[0,T]}$ with $T > 0$ is well-defined). For \(m \in \mathcal{M}\) and a compact set \(E \subset \mathbb{R}_+\), we define the total variation of \(m\) on \(E\) by:

\begin{equation}
    |m|(E) := \sup \left\{ \sum_{j=1}^N |m(E_j)| : \{E_j\}_{j=1}^N \text{ is a finite measurable partition of } E \right\}.
\end{equation}
\noindent We assume that the set of measure $m \in \mathcal{M}$ on \(\R_+\) is of locally bounded variation.

\smallskip 
\noindent $\bullet$ Convolution between a function and a measure. Let \(f : (0, T] \to \mathbb{R}\) be a measurable function and \(m \in \mathcal{M}\). Their convolution (whenever the integral is well-defined) is defined by
\begin{equation}\label{eq:convolmeasure}
	(f * m)(t)= \int_{[0,t)} f(t - s) \, dm(s) = \int_{[0,t)} f(t - s) \, m(ds) = (f\stackrel{m}{*}\mathbf{1})_t, \quad t \in (0, T].
\end{equation}

\smallskip 
\noindent $\bullet$ $X\perp \! \! \!\perp Y$  stands for independence of random variables, vectors or processes $X$ and $Y$.  

\noindent $\bullet$ For a measurable function \( \varphi: \mathbb{R}^+ \to \mathbb{R} \), \(\forall p \geq 1,\) we denote: 

\centerline{$
	  \| \varphi \|^p_{L^p([0,T])} := \int_0^{T} |\varphi(u)|^p \, du,  \; \displaystyle \|\varphi\|_{\infty}=\|\varphi\|_{\sup} := \sup_{u\in \mathbb{R}^+}|\varphi(u)| \; \text{and} \; \displaystyle \|\varphi\|_{\infty,T}=\|\varphi\|_{\sup,T} := \sup_{u\in [0,T]}|\varphi(u)|.
	$}

\smallskip 
\noindent $\bullet$ Let \([0,T]\) be a finite time horizon (\(T<\infty\)). 
Throughout the paper, a weak solution is understood to be defined on a filtered probability space $(\Omega,\mathcal F,\mathbb P,\mathbb F)$, where the filtration $\mathbb F=(\mathcal F_t)_{t\ge0}$ satisfies the usual conditions, namely it is right-continuous and $\mathbb P$-complete. The filtered probability space is assumed to support the following independent random elements:

\begin{itemize}
    \item a \( 2d\)-dimensional $\mathbb{F}$-Brownian motion $ \big(B, B^\top\big) = \big( (B_t)_{t \geq 0}, (B^\top_t)_{t \geq 0}\big)$ for \(d\geq1\) ;
    \item an independent $\mathbb{F}$-Poisson random measure $N$ on $(E, \mathcal{B}or(E))$ with compensator 
    $\xi(V_t, de)\,dt$, where the Levy measure $\xi$ is positive and satisfies
    $\xi(V_t, de) = \nu_0(de) + \sum_{i=1}^{d} V_t^i \, \nu_i(de)$ given a process $V$ 
    and a finite positive measure $\nu_k$ satisfying 
    \begin{align}
    \nu_k(\{0\}) = 0, \mbox{ for }k = 0,...,d .  
    \end{align}
\end{itemize}
\noindent
We denote by $\widetilde{N}(dt,de) := N(dt,de) - \xi(V_t, de)\,dt$ its compensated measure and $\mathcal{P}$ the $\sigma$-algebra of predictable processes, then, we define 
\begin{align*}
L^{\infty}_{\F}([0,T], \R^d) 
&= \left\{ Y:\Omega \times [0,T]\to \R^d,\; \F\text{-prog.~measurable and bounded a.s.} \right\} \\
L^p_{\F}([0,T], \R^d) 
&= \left\{ Y:\Omega \times [0,T]\to \R^d,\; \F\text{-prog.~measurable s.t.~} \E\Big[ \int_0^T |Y_s|^p ds \Big] < \infty \right\} \\
\mathbb{S}^{\infty}_{\F}([0,T], \R^d) 
&= \left\{ Y:\Omega \times [0,T]\to \R^d,\; \F\text{-prog.~measurable s.t.~} \sup_{t\leq T} |Y_t(\omega)|< \infty \text{ a.s.} \right\}\\
\mathbb{S}^{p}_{\F}([0,T], \R^d) 
&= \left\{ Y:\Omega \times [0,T]\to \R^d,\; \F\text{-prog.~measurable s.t.~} \E\Big[\sup_{0\leq t\leq T} |Y_t|^p\Big] < \infty \right\}\\
L^p_{\F}([0, T],\tilde{N}) 
&= \left\{ 
\begin{aligned}
U: \Omega \times [0,T]\times E \to \R,\;
\mathcal{P} \otimes \mathcal{B}or(E)\text{-measurable such that \qquad} \\
 \mathbb{E} \left[\int_0^T \int_E |U_t(e)|^p \, \xi(V_t, de) \, dt \right] < \infty\qquad \qquad\qquad
\end{aligned}
\right\}.
\end{align*}
In the following, unless otherwise specified, $V$ denotes the Volterra volatility defined in \eqref{VolSqrt_}.
Here $|\cdot|$ denotes the Euclidian norm on $\R^d$, \(d\geq1\) for short.  Classically, for $p \in (1, \infty  )$, we define $L^{p, loc}_{\F}([0,T], \R^d)$ as the set of progressive 
processes $Y$ for which there exists a sequence of increasing stopping times 
$\tau_n \uparrow \infty$ such that the stopped processes $Y^{\tau_n}$ are in $L^{p}_{\F}([0,T], \R^d)$ for every $n \geq 1$, and we recall that it consists of all progressive processes $Y$ s.t. 
$ \int_0^T |Y_t|^p dt$ $<$ $\infty$, a.s. Likewise for $\mathbb{S}^{p, loc}_{\F}([0,T], \R^d)$. To unclutter notation, we write $L^{p, loc}_{\F}([0,T])$  instead of $L^{p, loc}_{\F}([0,T], \R^d)$ when 
the context is clear. 

\smallskip 
\noindent $\bullet$  We will use the matrix norm \(|A| = \operatorname{tr}(A^{\top}A)\)
in this paper whenever $A\in \mathbb{R}^{d\times d}$ with $d\geq1$.

\smallskip 
\noindent $\bullet$ Let $M$ be a c\`adl\`ag local martingale vanishing at $0$. The stochastic
exponential $\mathcal{E}(M)$ is defined by
\[
    \mathcal{E}(M)_t
    = \exp\!\left(M_t - \tfrac{1}{2}\langle M^c \rangle_t\right)
    \prod_{0 < s \leq t}\bigl(1 + \Delta M_s\bigr)\exp(-\Delta M_s),
    \quad t \in [0,T],
\]
where $M^c$ denotes the continuous martingale part of $M$,
$\langle M^c \rangle$ its quadratic variation, and $\Delta M_s := M_s - M_{s-}$
the jump of $M$ at time $s$. When $\Delta M_t > -1$ for all $t \in [0,T]$
$\mathbb{P}$-a.s., $\mathcal{E}(M)$ is a positive local martingale thus a supermartingale.

\medskip
\noindent Our problem will be formulated under a given complete filtered probability space \((\Omega,\mathcal{F},\mathbb{F} = \{\mathcal{F}_t\}_{0\leq t\leq T}, \P)\). The filtration \(\mathbb{F}\) is not necessarily the augmented filtration generated by \( (B, B^\top,N)\) ;
thus, it can be a strictly larger filtration. Here \(\P\) is a real-world probability measure from which a family of equivalent probability measures can be generated. 

\section{Preliminaries: Multivariate affine Volterra models with jumps}\label{sect-preliminaries}
\subsection{Volterra processes with convolution kernels and jumps.}\label{sect-prelim}
\noindent Fix $T > 0$, $d\in \N$.
We let $K=\diag(K_1,\ldots,K_d)$ be diagonal with scalar kernels $K_i\in L^{2}([0,T],\R)$ on the diagonal, $\varphi=\diag(\varphi^1,\ldots,\varphi^d)$, $\sigma^v=\diag(\sigma^v_1,\ldots,\sigma^v_d)$, $\varsigma=\diag(\varsigma^1,\ldots,\varsigma^d)$ with \(\varsigma^i\) a (locally) positive bounded Borel 
function and $D$ a $\R^{d\times d}$-valued matrix satisfying
\[
D_{ij} \ge 0 \ \text{for } i \neq j.
\]

Denote by $\mu : [0,T] \to \mathbb{R}^d_{+}$ a positive function. Let $\eta = (\eta^1,\ldots,\eta^d)^\top$ be a measurable map $\eta : E \to \mathbb{R}^d$.
Let $V=(V^1,\ldots, V^d)^\top$  be the following $\R^d_+$--valued scaled Volterra square--root process driven by an $d$-dimensional process $W=(W^1,\ldots,W^d)^\T$ and the compensated random measure $\widetilde{N}$
:
\begin{equation} 
	\label{VolSqrt_}
	\begin{aligned}
		V_t = \varphi(t) V_0 + \int_0^t K(t-s) \Big( \big(\mu(s) + D V_s\big) ds  + \sigma^v \varsigma(s)\sqrt{\diag(V_s)}dW_s + \int_{E}\eta(e)\tilde{N}(ds,de) \Big), \quad V_0\perp\!\!\!\perp W \perp\!\!\!\perp \widetilde{N}.
	\end{aligned}
\end{equation}
Here $\eta$ specifies how a mark $e \in E$ 
generated by the Poisson random measure $N$ is translated into a jump $\eta(e)$ in 
each component of the vector of volatility. The last two terms in~\eqref{VolSqrt_} form an $\R^d-$valued, purely discontinuous local martingale. We next discuss conditions under which the inhomogeneous stochastic Volterra square-root equation with jumps~\eqref{VolSqrt_} admits a solution satisfying the moment estimate
\begin{equation}\label{eq:moments V1}
    \sup_{0\leq t\leq T} \E\left[|V_t|^p\right]
    \leq C\left(1+\E\left[|V_0|^p\right]\right)<\infty,
    \qquad p>0,
\end{equation}
for some constant $C>0$.  Throughout, we work under the following assumption.


\begin{assumption}[On Volterra Equations with convolutive kernels and jump]\label{assump:kernelJumpVolterra}
	\begin{enumerate}
		\item[(i)] 
	Assume that \(K\) is diagonal with scalar kernels $K_i\in L^2_{\mathrm{loc}}(\mathbb{R}_+)$ on the diagonal for \(i=1,\ldots,d\) that are completely monotone on $(0, \infty)$, not identically zero, 
     and there exists $\vartheta_i\in(0,1]$ such that for every $0\leq s < t\leq T <\infty $
\begin{equation}
(t-s)^{-2\vartheta_i}
\Big(
\int_s^t K_i(t-u)^2\,du
+
\int_0^T
\big(K_i((t-s)+u)-K_i(u)\big)^2\,du
\Big)
<\infty.
\end{equation}
	\noindent
	\item[(ii)]  Assume that, for all $i = 1, \ldots, d$ and all $k = 0, \ldots, d$, the jump sizes satisfy
\[
\nu_k\big(\{ e \in E : \eta^i(e) < 0 \}\big) = 0,
\]
and are square-integrable i.e.,
\[
\int_E \big(\eta^i(e)\big)^2 \, \nu_k(de) < \infty.
\]

        \noindent
		\item[(iii)]
        Finally, assume that \( V_0^i \in L^p(\mathbb{P}) \) for some suitable \( p \in (0, +\infty) \), such that
		the process $t \to v_0^i(t) =V_0^i \varphi^i(t)$ is strictly positive and  absolutely continuous and $(\mathcal F_t)$-adapted.
		Moreover, for some $\delta_i > 0$, for any $p > 0$, there exists \(C:=C_{T,p}\) such that

		\begin{equation*}
		\mathbb{E} \,\!\Big(\sup_{t \in [0,T]} |v_0^i(t)|^p\Big) < +\infty,\quad
		\mathbb{E}\!\big[\,|v_0^i(t') - v_0^i(t)|^p\,\big] 
		\le C \Big( 1 + \mathbb{E}\,\big[\sup_{t \in [0,T]} |v_0^i(t)|^p\big] \Big) |t' - t|^{\delta_i p}.
		\end{equation*}
	\end{enumerate}
\end{assumption}
\noindent
The affine structure of the drift coefficient $b(t,x) = \mu(t)+ Dx$ in~\eqref{VolSqrt_},
immediately yields its global Lipschitz continuity with respect to the state variable, uniformly over $t\in[0,T]$. In addition, the drift term \(b\) and diffusion coefficient \(\sigma(t, x):=\sigma(t)\sqrt{\diag(x)}\), together with the jump characteristics $(\eta,\pi)$ under condition~\ref{assump:kernelJumpVolterra}~$(ii)$, satisfy the standard linear growth condition. More precisely, there exists a constant $C_{b,\sigma,\eta}>0$ such that
\begin{equation}\label{eq:LinearGrowth}
	|b(t, x)| + \|\sigma(t, x)\| + \int_E |\eta(e)|^2\,\pi(x,de) \leq C_{b,\sigma, \eta}(1 + |x|),
	\quad \text{for every } (t,x)\in[0,T]\times\mathbb{R}^d.
\end{equation}
\begin{Remark}\label{rm:Kernels}
	\noindent 1. Assumption~\ref{assump:kernelJumpVolterra}~$(i)$ covers, for instance,  constant non-negative kernels,  fractional kernels  of the form $K_{\alpha}(t) =\frac{t^{\alpha-1}}{\Gamma(\alpha)}\mathbf{1}_{\mathbb{R}_+}(t)$ with $\alpha  \in(\frac12,1]$, exponentially decaying kernels ${\rm e}^{-\beta t}$ with $\beta>0$ and more generally the the exponentially decaying fractional kernels  \(K_{\alpha,\beta}(t) = \frac{t^{\alpha-1}}{\Gamma(\alpha)} e^{-\beta t}\mathbf{1}_{\mathbb{R}_+}
	\) with \( \alpha \in \left( \tfrac{1}{2}, 1 \right] \) and \( \beta \ge 0 \) ( see e.g. 
	\cite{EGnabeyeu2025,GnabeyeuPages2026b}, \cite[ Example 2.2 ]{GnabeyeuPages2026}). These kernels satisfy $(\widehat {\cal K}^{cont}_{\widehat \theta})$ and $({\cal K}^{cont}_{\theta})$, for $\alpha >1/2$ with $\vartheta = \min\bigl( \alpha-\frac12,\; 1\bigr)$.
	
	\noindent 2. The roughness of the variance paths is determined by the parameter $\alpha$ linked to the Hurst parameter $H$ via the relation $\alpha=H+\frac{1}{2}$.
	For $\alpha\rightarrow 1$ we recover the classical markovian square root process.

\end{Remark}
The existence of a solution $V$ to~\eqref{VolSqrt_} satisfying the bounded moment estimate~\eqref{eq:moments V1} is guaranteed under suitable specifications of $\varphi$, as detailed below for the one-dimensional case $d=1$:
\begin{itemize}[leftmargin=*, labelsep=0.5em] 
    \item \noindent In the case of continuous stochastic Volterra square root equations, i.e. when \(\eta\equiv0\) and the kernel is of fractional type (corresponding to \(K=K_{\alpha}\) with \(\alpha\in [\frac12,1)\)),  it follows from~\cite{EGnabeyeuR2025} 
    that Equation~\eqref{VolSqrt_} admits at least a unique (in law) positive weak solution as a scaling limit of a sequence of 
time-modulated Hawkes processes with heavy-tailed kernels in a nearly unstable regime. Its multidimensional extension $d\geq 1$ is straightforward at least when \(D\) is diagonal, as the argument can be carried out component by component.
Moreover, under assumption~\ref{assump:kernelJumpVolterra}~$(i)$ and~$(iii)$, each component \( t \mapsto V_t^i \) of a solution to  Equation~\eqref{VolSqrt_}  starting from  $V_0^i$ 
has a \( \big( \delta_i \wedge \vartheta_i - \epsilon \big) \)-H\"older pathwise continuous modification  on $\R_+$ for sufficiently small \( \epsilon > 0 \).
\item Under Assumption \ref{assump:kernelJumpVolterra}, and the linear growth~\eqref{eq:LinearGrowth}, when \(\varphi=\varsigma=1\)  and $V_0 \in \R_+$, ~\cite[Theorem 2.7. or Theorem 5.2.]{AlfonsiSzulda2025} establishes the strong existence and pathwise uniqueness of a non-negative c\`adl\`ag solution of the stochastic Volterra equation with jumps (\ref{VolSqrt_}) such that~\eqref{eq:moments V1} holds. 
\end{itemize}


\subsection{Formulation of the stochastic Market model}
\noindent We  consider a frictionless financial market on $[0,T]$  
with \(d+1\) securities, consisting of a  bond and \(d\) stocks. The non--risky asset  $S^0$ satisfies the (stochastic) ordinary differential equation: 
\begin{align*}
	dS^0_t = S^0_t r(t) dt,
\end{align*}
with a time-dependent deterministic  short risk-free rate $r:\R_+ \to \R$, and  $d$ risky assets (stock or index) whose return vector process $(S_t)_{t \ge 0} = (S_{t}^1, \ldots, S_{t}^d)_{t \ge 0}$ is defined via the dynamics given by the vector-stochastic differential equation (SDE):
\begin{align}
	\label{eq:stocks}
	dS_t = \diag(S_t) \big[ \big( r(t) {\bold{1}_d} + \sigma_t \lambda_t  \big)dt + \sigma_t dB_t \big],
\end{align}
driven by a $d$-dimensional Brownian motion $B$, with a $d\times d$-matrix valued  continuous stochastic volatility process $\sigma$ whose dynamics is driven by~\eqref{VolSqrt_}
and a $\R^d$-valued continuous stochastic process $\lambda$, 
called {\it market price of risk}.  Here ${\bold{1}_d}$ denotes the vector in $\R^d$ with all components equal to $1$ and the correlation structure of $W$ with $B$ is given by
{\begin{align}\label{eq:correstructureheston}
		W^i = \rho_i B^i + \sqrt{1-\rho_i^2} B^{\perp,i} =  \Sigma_i^\top  B_t + \sqrt{1-\Sigma_i^\T \Sigma_i} B^{\perp,i}_t, \quad i=1,\ldots,d,
	\end{align}
	for some $(\rho_1,\ldots,\rho_d)\in[-1,1]^d$},
where $(0,\ldots,\rho_i,\ldots,0)^\T:=\Sigma_i \in \R^{d}$ is such that $\Sigma_i^\T \Sigma_i\leq1$,  and $B^{\perp}$ $=$ $(B^{\perp,1},\ldots,B^{\perp,d})^\T$ is an $d$--dimensional Brownian motion independent of $B$. The correlation $\rho_i $ between stock price \(S^i\) and variance \(V^i\) is assumed constant.
Note that $d\langle W^i \rangle_t = dt$  but $W^i$ and $W^j$ can be correlated, hence $W$ is not necessarily a Brownian motion.

\medskip
\noindent Observe that 
processes $\lambda$ and $\sigma$ are $\F$-adapted, possibly unbounded,  but not necessarily adapted to the filtration generated by $W$.  We point out that $\F$ may be strictly larger than the augmented filtration generated by $B$ and $B^{\perp}$ as we deal with weak solutions to stochastic Volterra equations. 

\medskip
\noindent
We assume that  $\sigma$ in~\eqref{eq:stocks} is given by $\sigma = \sigma(V) = \sqrt{\diag(V)}$, where the $\R^d_+$--valued scaled process $V$  is defined in~\eqref{VolSqrt_}. We will be chiefly interested in the case where \(\lambda_t\) is linear
in \(\sigma_t\). More specifically, the market price of risk (risk premium) is assumed to be in the form
{$\lambda$ $=$ $\big(\theta_1\sqrt{V^1},\ldots,\theta_d \sqrt{V^d}\big)^\top$}, for some constant {$\theta_i \geq 0$},  so that the dynamics  for the stock prices \eqref{eq:stocks} reads following \cite{Kraft2005,Gnabeyeu2026b}
\begin{align}
	\label{eq:hestonS}
	dS^i_t = S^i_t \left( r(t)  + \theta_i V^i_t   \right) dt + S^i_t \sqrt{V^i_t} dB^i_t, \quad i=1,\ldots, d.
\end{align}
Since \(S\) is fully determined by \(V\), the existence of $S$ readily follows from that of $V$. In particular, weak existence of c\`adl\`ag solution $V$ of~\eqref{VolSqrt_} is established  under suitable assumptions on the kernel $K$ and
 the jump components $(\eta,\xi)$ (see, e.g., Assumption~\ref{assump:kernelJumpVolterra} and the discussion that follows).
 we may therefore assume that the stochastic Volterra equation (\ref{eq:hestonS})-(\ref{VolSqrt_})  admits a unique in law 
c\`adl\`ag $\R^{d}_+ \times \R^{d}_+$-valued weak solution $(S,V)$ for any initial condition $(S_0, V_0) \in \R^{d}_+ \times \R^{d}_+$ defined on some filtered probability space $(\Omega,\mathcal F, (\mathcal F)_{t\geq 0}, \mathbb P)$ such that $V$ satisfied~\eqref{eq:moments V1}.

\noindent From now on, we set \( g_0(t):= \varphi(t) V_0 + \int_0^t K(t-s)\mu(s) ds \) and \(Z_t := \int_{0}^{t}\big(D V_s ds + \sigma^v\varsigma(s) \sqrt{\diag(V_s)}dW_s + \int_{E}\eta(e)\tilde{N}(de,ds) \big) ,\) for all \( t \geq 0\) with $\sigma(V) = \sqrt{\diag(V)}$ so that Equation~\eqref{VolSqrt_} reads
\begin{equation} 
	\label{VolSqrt2} 
	\begin{aligned}
		V_t = g_0(t) + \int_0^t K(t-s)\Big(  D V_s ds  + \sigma^v \varsigma(s)\sigma(V_s)dW_s + \int_{E}\eta(e)\tilde{N}(ds,de) \Big) =  g_0(t) + \int_0^t K(t-s) dZ_s.
	\end{aligned}
\end{equation}
Here \(Z\) is a $d-$dimensional semimartingale starting at $0$.
Finally, we consider the $\R^d$-valued process for \( s\geq t,\) 
\begin{align}\label{eq:processg}
	g_t(s)= g_0(s) + \int_0^t K(s-u) \big(D V_u du + \sigma^v\varsigma(u) \sigma(V_u)dW_u + \int_{E}\eta(e)\tilde{N}(du,de)\big) =  g_0(s) + \int_0^t K(s-u) dZ_u.
\end{align}
One notes that for each, $s\leq T$, $(g_t(s))_{t\leq s}$ is the adjusted forward process 
\begin{align}\label{eq:Condprocessg}
	g_t(s) &= \; \mathbb E\Big[  V_s - \int_t^s K(s-u)DV_udu \Mid \cF_t\Big].
\end{align}
This adjusted forward process is commonly used (see, e.g.,~\cite{Gnabeyeu2026a, Gnabeyeu2026b}) to elucidate the affine structure of affine Volterra processes with càdlàg trajectories.\\
\noindent The process in~\eqref{VolSqrt2} is non-Markovian and non-semimartingale in general. Note that our model (\ref{eq:hestonS})-(\ref{eq:correstructureheston})-(\ref{VolSqrt_}) features correlation between the stocks and between a stock and its volatility. 
This also provides an extension to the inhomogeneous setting and jumps of the models considered in~\cite{  TomasRosenbaum2021, AichingerDesmettre2021, Gnabeyeu2026a, Gnabeyeu2026b}.


\section{The Merton's utility maximization problem and related wealth}\label{sect-MertProblem}
\noindent {$\rhd$ {\em Preliminaries and Problem formulation}:}  
Since we work with weak solutions to the stochastic Volterra equations with jumps \eqref{eq:hestonS}--\eqref{VolSqrt_}, the Brownian motions and the Poisson random measure are part of the solution. As 
the expected utility depends only on the law of the wealth process, we may fix, without loss of generality, a weak solution $(S,V,B,B^\top,N)$ to \eqref{eq:hestonS}--\eqref{VolSqrt_}, defined on the filtered probability space $(\Omega,\mathcal F,\mathbb F,\mathbb P)$, 
where $\mathbb F=(\mathcal F_t)_{t\ge0}$ satisfies the usual conditions.

\smallskip
\noindent
We consider the classical problem of maximizing the expected utility of terminal wealth with particular emphasis on exponential , power, and logarithmic utilities. One introduces $\tilde{\mathcal{A}}$ the set of all $\mathbb{F}$-progressively measurable processes $ (\alpha_t)_{t \in [0,T]}$ valued in the Polish space  \(\R^d\).
The investor's goal in the Merton problem is to find an optimal strategy so as to maximize the expected utility of terminal wealth. Specifically, given a utility function $U$ on a subset of $\R$ and starting from an initial capital $x_0$,
the objective of the agent is

\begin{equation}\label{eq:value0}
	\mathcal{V}(x_0,V_0):= \sup_{\alpha(\cdot) \in \mathcal A}
	\mathbb E\!\left[ U\!\left(X_T^\alpha\right) \right],  \;\text{given} \; x_0 \;\text{and} \;V_0.
\end{equation}
with $X^\alpha$ the wealth
controlled by $\alpha \in \mathcal A$, starting from $x_0$ at time $0$
and $\mathcal A$ the subset of controls $\alpha \in \tilde{\mathcal{A}} $ such that the family \(\left\{ U\!\left(X_\tau^\alpha\right) : \tau \in [0,T] \right\}\)
is uniformly integrable. 

The idea is that an admissible control (or strategy) is optimal\footnote{Optimal strategy: No expected gain from deviation, this implies martingality} if the associated value process is a martingale and for any other admissible control, it is a supermartingale. That is the classic martingale \textit{optimality principle}, see, e.g., \cite{HuImkellerMueller2005}, \cite[Section 6.6.1]{pham2009continuous} ,\cite{JeanblancEtAl2012} or more recently \cite{turki2026portfolio} when we deal with signal on the jumps. 

\begin{Definition}[Martingale optimality principle]\label{def:Martopt}
	The Problem (\ref{eq:value0}) can be solved by constructing a family of processes $\{ J^\alpha_t \}_{ t \in [0, T]}$, $\alpha \in \cA$, satisfying the conditions:
	\begin{enumerate}
		\item $J^\alpha_T = U(X^\alpha_T)$ for all $\alpha \in \cA$; 
		\item $J^\alpha_0$ is a constant, independent of  $\alpha \in \cA$;
		\item $J^\alpha$ is a supermartingale for all $\alpha \in \cA$, and there exists  $\alpha^* \in \cA$ such that $J^{\alpha^*}$ is a martingale.
	\end{enumerate}
\end{Definition}
\noindent A family of processes with the above properties can now be used to compare the expected utilities of an arbitrary strategy $\alpha \in \cA$ and the strategy $\alpha^*$:
\begin{equation*}
	\E[ U(X_T^\alpha) ] = \E[ J^\alpha_T ] \leq J^\alpha_0 = J^{\alpha^*}_0 = \E[ J^{\alpha^*}_T]  = \E[ U(X^{\alpha^*}_T)]=\mathcal{V}(x_0,V_0).
\end{equation*}
where $X^{\alpha^*}$ is the wealth process under $\alpha^*$. Thus the strategy $\alpha^*$ is indeed our desired optimal portfolio strategy.

\medskip
\noindent We will offer explicit solutions to the optimal portfolio policies that depend on a function  $\psi$, solution of a multivariate Riccati-Volterra equation in the spirit of~\cite{Gnabeyeu2026b}.  
In particular, as shown in Theorem~\ref{Thm:VolSqrtAll}, the Laplace transform of the integrated variance process \(V\) admits a semi-closed-form exponential--affine representation. This representation is expressed in terms of the (continuous) solution $\psi$, when it exists, to an associated time-inhomogeneous Riccati-type Volterra equation. It is then used to obtain an explicit characterization of the ansatz processes $({ J^\alpha_t }_{t \in [0,T]})$.
Let $\Lambda$ and $U$ be defined as 
\begin{equation}\label{eq:LambdaU}
	\Lambda_t^i:=  \sigma^v_i\varsigma^i(t) \psi^i(T-t) \sqrt{V^i_t}, \; i=1,\ldots,d, \; \text{and}\; U_t(e) :=\psi^\top(T-t) \eta(e) = \sum_{i=1}^d\psi^i(T-t) \eta^i(e), \; e\in E, \; 0 \leq t \leq T.
\end{equation}
We will work under the following standing assumption,
\begin{assumption}\label{assm:gen}
	Assume that there exists a solution
	$\psi \in C([0,T],\mathbb{R}^d)$ to the above-mentioned inhomogeneous Riccati--Volterra equation 
    such that the below exponential integrability condition holds true
\begin{equation}\label{eq:cond_U}
    \int_E \exp\!\Big( p \sup_{t \in [0,T]} |U_t(e)| \Big)\, \nu_i(de) < \infty,
\end{equation}
for some $p > 0$ sufficiently large and for all $i = 0, \ldots, d$. Moreover, we define for each $i = 1, \ldots, d$,
\begin{equation}\label{eq:beta_i}
\beta_i := \int_E 
\Big(1 + \exp\!\big(2 \sup_{t \in [0,T]} |U_t(e)|\big)\Big)
\, \nu_i(de).
\end{equation}
 We furthermore assume that the solution
	$\psi \in C([0,T],\mathbb{R}^d)$  
    satisfies the below appropriate
	boundedness condition i.e. is such that 
	\begin{equation}\label{eq:condtheta}
		\max_{1 \leq i \leq d} \Big[ \sup_{t \in [0,T]} \left( \theta_i^2 + (\sigma^v_i)^2 \varsigma^i(t)^2 \psi^i(T-t)^2 \right) + \beta_i \Big]\leq \frac{a}{a(p)},
	\end{equation}
	holds for a sufficient large $p > 1$, where the constant $a(p)$ is given by 
	{  \begin{equation}a(p)=\max \Big[p \left(2 + |\Sigma| \right),   {2 (8p^2 {- 2p}) \left( 1  + |{\Sigma}|^2  \right)}, { p \left( 1  + |{\Sigma}|^2  \right)}, \,{ 4 p }  \Big]. \label{eq:constap}
	\end{equation}}
	and the constant $a>0$ is such that $\E\left[\exp\big(a\int_0^T \sum_{i=1}^d V^i_s ds\big)\right] < \infty$.\\

\end{assumption}


\begin{Remark}[on Assumption~\ref{assm:gen}:]\label{rm:ass_} 
1. The first condition~\eqref{eq:cond_U} on $U$ is particularly useful for establishing the martingale property in Lemma~5.1. A similar condition appears in~\cite{richter2014explicit} in~\cite{richter2014explicit} (see proposition  4.2 in section~4.1). 
If $U_t$ is non-positive for every $t\in [0,T]$, then condition~\eqref{eq:cond_U} can be relaxed to 
\begin{equation}\label{eq:Relax_cond_U}
\sup_{t \in [0,T]}\int_E
|U_t(e)|^2
\nu_i(de)<\infty,\;\text{i.e.} \; U\in L^\infty\big([0,T];L^2_{\nu_i}(E,\R)\big) \;\text{for every }\; i = 0, \ldots, d.
\end{equation}
In this case, we instead set $\beta_i := \sup_{t \in [0,T]}\int_E  |U_t(e)|
\, \nu_i(de),$ for each $i = 1, \ldots, d$, which are well defined and finite by the Cauchy--Schwarz inequality and the finiteness of $\nu_i(E)$ for each $i=0,\ldots,d$. See e.g., the proof of Lemma~\ref{lm: extended_m_EG_lemma} together with the subsequent remark.  Note that~\eqref{eq:Relax_cond_U} is in particular satisfied by $U$ in~\eqref{eq:LambdaU} whenever $\psi\in C([0,T],\mathbb{R}_-^d)$ and Assumption~\ref{assump:kernelJumpVolterra}~(ii) holds.

\medskip
\noindent 2. The constants \(\beta_i\), $i = 1, \ldots, d$ in~\eqref{eq:beta_i} are well defined and finite due to the finiteness of $\nu_i$ for each $i = 1, \ldots, d$.
Note that if Assumption~\ref{assm:gen} hold, then letting  \(\beta:=(\beta_1, \ldots,\beta_d)^\top\),
\begin{equation}\label{eq:assumption_novikov}
	\E \Big[ \exp\Big( a(p)\int_0^T \big(  |\lambda_s|^2 + \left|\Lambda_s\right|^2 + \beta^\top V_s\big)ds \Big) \Big] \;< \;  \infty,
\end{equation}
holds for some $p > 1$ and a constant $a(p)$ given by~\eqref{eq:constap}.

\noindent In fact, under Assumption~\ref{assm:gen}, we will have 
\begin{equation}
	a(p)\left( |\lambda_s|^2 + \left|\Lambda_s\right|^2 + \beta^\top V_s \right) \; = \; a(p) \sum_{i=1}^d  \left( \theta_i^2 + (\sigma^v_i)^2 \varsigma^i(s)^2\psi^i(T-s)^2 + \beta_i \right)\,V_s^i \; \leq \;  a \sum_{i=1}^d V_s^i,
\end{equation} 
which implies that $\E \left[ \exp\Big(  a(p) \int_0^T \Big( |\lambda_s|^2 + \left|\Lambda_s\right|^2 + \beta^\top V_s \Big)ds\Big) \right]< \infty$, thanks to the exponential-affine representation in Theorem~\ref{Thm:VolSqrtAll} ( with \(\mathcal{M} \ni m(\dd s) := a{\bold{1}_d}\,\delta_0(\dd s) \)).
\end{Remark}

\medskip
\noindent{\bf Remark:} 
Condition \eqref{eq:condtheta} concerns the risk premium constants  $(\theta_1,\ldots, \theta_d)$. For a large enough constant $a>0$, from Theorem~\ref{Thm:VolSqrtAll} ( with \(\mathcal{M} \ni m(\dd s) := a{\bold{1}_d}\,\delta_0(\dd s) \)),   a sufficient condition ensuring that
$\E\big[\exp\big(a\int_0^T \sum_{i=1}^d V^i_s ds\big)\big]<\infty$ is the existence of a continuous solution $\tilde{\psi}$  on $[0,T]$ to the inhomogeneous Riccati--Volterra equation with L\'evy exponential jump compensator \(\forall t \in [0,T]\) and $i=1,\ldots,d$,
\begin{align} 
	\tilde{\psi}^i(t) &= \; \int_0^t K_i(t-s) \Big(a+ \big(D^\T\tilde{\psi}(s)\big)_i + \frac{(\sigma^v_i)^2}{2}(\varsigma^i(T-s)\tilde{\psi}^i(s)) ^2\Big) ds  + \int_{E}h_1\big(\psi(T-s)^\top\eta(e)\big)\nu_i(de). 
\end{align} 
where  $h_\lambda$ is the scaled compensated exponential function defined for every $\lambda>0$ by
\begin{equation}\label{eq:ScaledExpCom_funct}
h_\lambda(x) = \frac{e^{\lambda x} - \lambda x - 1}
{\lambda}, \quad \text{for all } x \in \mathbb{R}.
\end{equation}

\noindent Exponential, power and logarithmic utility functions are widely adopted in the literature and display distinct risk-aversion properties. Specifically, the power utility function exhibits constant relative risk aversion (CRRA), 
whereas the exponential utility function is characterized by constant absolute risk aversion (CARA). Optimal strategies corresponding to these utility functions are presented in the sequel, within the framework of the model introduced in Section~\ref{sect-prelim}.
\subsection{Optimal strategy for the exponential utility maximization problem.}\label{subsect:Expo}
\noindent In this subsection, we consider the exponential utility case.
Here, let $\pi_t$ denote the vector of the amounts invested in the risky assets $S$ at time $t$ in a self--financing strategy. We assume that the the process \((\pi_t)_{t\geq0}\) are progressively measurable. Then, the dynamics of the wealth $X^{\pi}$ of the portfolio we seek to optimize is given by 
\begin{align*}
	\mathrm d X^\pi_t
	&=  \bigl(  \pi_t^\top \big( r(t) {\bold{1}_d} + \sigma(V_t) \lambda_t  \big) + (X^\pi_t - \pi_t^\top {\bold{1}_d})r(t) \bigr)\,\mathrm dt + \pi_t^\top\sigma(V_t)  \,\mathrm d B_t \\
	&=  \bigl(r(t) X^\pi_t + \pi_t^\top  \sigma(V_t) \lambda_t\bigr)\,\mathrm dt + \pi_t^\top   \sigma(V_t) \,\mathrm d B_t.
\end{align*}
Letting $\alpha_t := \sigma^\top(V_t)\pi_t $ be the investment strategy, the wealth $X^{\alpha}$ reads:
\begin{align}
	\label{eq:wealth}
	dX^{\alpha}_t &= \big( r(t) X^{\alpha}_t  + \alpha_t^\T \lambda_t \big) dt + \alpha_t^\T dB_t, \quad t \geq 0, \quad X_0^\alpha = x_0 \in \R. 
\end{align} 
By a standard calculation, the wealth process is then given by
\begin{equation}\label{eq:wealthProcess} 
	X_t = e^{ \int^t_0 r(s) ds} \Big( x_0 + \int_0^t e^{ -\int^s_0 r(u) du}  \Big(\alpha_s^\T dB_s + \alpha_s^\T \lambda_s \mathrm{d}s \Big) \Big). 
\end{equation}
Note that it is sufficient to assume that \(\int_0^t (\left|\lambda_s\right|^2 + \left|\alpha_s\right|^2 )\, \mathrm{d}s < +\infty \) almost surely for all \( t \ge 0\)
in order to construct the stochastic integrals in Equation~\eqref{eq:wealthProcess}. This boundedness condition holds owing to the inequality \( |z|^2 \le 2 e^{|z|}, \) for all $z \in \R $, together with the condition~\eqref{eq:assumption_novikov}, and the following admissibility assumption, which is consistent with \cite[Definition~3.8]{Gnabeyeu2026b}. 
	The set of all admissible investment strategies is denoted as $\cA$ and is naturally defined by
    \begin{equation}\label{eq:AdmStrat2}
	\mathcal A  = \left\{ 
	\begin{aligned}
		& (\alpha_t)_{t\in [0,T]}\in {L^{2,loc}_{\F}([0,T], \R^d)} \mbox{ such that \eqref{eq:wealth} admits a  solution }~ X_t^{\alpha} \\ 
		& \text{~for which~} \Big(\exp\big[ - \gamma e^{ \int^T_t r(u) du} X_t^{\alpha} \big]\Big)_{t\in[0,T]} \text{~is of class } D.
	\end{aligned}
	\right\}
	\end{equation}
Recall that a process is said to be of class D if the family
 \[\Big\{\exp\big[ - \gamma e^{ \int^T_\tau r(u) du} X_\tau^{\alpha} \big]: \tau \text{ stopping time valued in } [0, T] \Big\}\] is uniformly integrable.
 
\noindent The investor now considers the Merton exponential utility optimization problem, i.e. its aim is to find the value function $\mathcal{V}^\xi(x_0,V_0)$ for the CARA utility function such that
\begin{equation}\label{Expobj}
	\mathcal{V}^\xi(x_0,V_0) = \sup_{\alpha(\cdot) \in \cA } \E_{x_0,V_0}\left[U(X_T^\alpha-\xi) \right],
\end{equation}
where \(U(x):=-\frac{1}{\gamma}\exp(-\gamma x )\), $\xi$ is a $\mathcal{F}_T$-measure terminal condition affine in $\int_0^TV_sds$ (can be interpreted as a variance swap product) and $\gamma >0 $ the risk aversion coefficient of the investor .\\

\noindent From the dynamic of the controlled wealth process given in Equation~\eqref{eq:wealthPowerSol} 
we have \begin{equation}\label{eq:LogUtil}
		U(X_T^\alpha-\xi) =
		-\frac{1}{\gamma}\exp(-\gamma x_0  e^{ \int^T_0 r(s) ds}) \exp\!\left(-\gamma \Big(\int_0^T e^{ \int^T_s r(u) du}  \Big(\alpha_s^\T dB_s + \alpha_s^\T \lambda_s \mathrm{d}s \Big)-\xi\Big)
	\right).
	\end{equation}
Similar to~\cite{HuImkellerMueller2005},
	we consider the family of stochastic processes \((J^\alpha)^\alpha\)  defined for every \( t\in[0,T]\) by \begin{equation}\label{eq:ansatzExpo}J_t^\alpha:=
		-\frac{1}{\gamma}\exp(-\gamma x_0  e^{ \int^T_0 r(s) ds}) \exp\!\left(-\gamma \int_0^t e^{ \int^T_s r(u) du}  \Big(\alpha_s^\T dB_s + \alpha_s^\T \lambda_s \mathrm{d}s \Big) + \gamma Y_t \right) 
    \end{equation}where the triplet \((Y,\Lambda, U)\) satisfies a Backward SDEs with jumps (BSDEJ for short) 
    under \(\P\) with a driver \(f: [0,T] \times \mathbb{R} \times \R^d \times \mathbb{R} \to \R\) of the form:
	\begin{equation}\label{eq:GammaDef2}
		\left\{
		\begin{array}{ccl}
			dY_t &=&  \; -f(t,Y_t ,\Lambda_t, U_t)dt + \Lambda_t^\top dW_t + \int_{E}U_t(e)\tilde{N}(dt,de), \\[1mm]
			Y_T &=& \Xi =\int_0^T \zeta V_s ds =: \int_0^T \sum_{i=1}^d \zeta_i V^i_s ds. 
		\end{array}  
		\right.
	\end{equation}
    for some $ \zeta := (\zeta_1,\cdots, \zeta_d) \in \R^d$ such that $\zeta_i\leq\frac{\theta_i^2}{2\gamma}$ for $i=1,\cdots,d$.   We refer the reader to~\cite{Delong2013} for a detailed presentation of the theory of BSDEJ. 
    In these terms we are bound to choose a function \(f\) for which \(J_t^\alpha\) is a
	supermartingale for all \(\alpha\in \mathcal A\) and there exists a \(\alpha^*\in \mathcal A\) such that \(J_t^{\alpha^*}\) is a martingale. 

    \medskip
    \noindent We start by the below proposition establishing that, for some \(p>1\), the triplet $\left(Y,\Lambda, U\right)$  is a $\mathbb{S}^{p}_{\F}([0,T], \R) \times L^2_{\F}([0,T], \R^d)\times L^2_{\F}([0,T], \tilde{N})$-valued solution to a Riccati BSDE with jumps provided solvability in \(C([0,T],\R^d \setminus\{0\})\) of an inhomogeneous Ricatti-Volterra equation.
\begin{Proposition}
	\label{prop:ExpoExp_riccati} 
	Assume that there exists  a solution $\psi \in C([0,T],\R^d)$ to the inhomogeneous integral Riccati-Volterra equation with L\'evy exponential jump compensator:
	\begin{align}
		\psi^i(t)&= \int_0^t K_i(t-s)\big(\zeta_i-\frac{ \theta_i^2}{2\gamma}  + F_i(T-s,\psi(s))\big) ds, \quad i=1,\ldots,d,  
		\label{eq:RiccatiExpTilpsi1} \\
		F_i(s,\psi) &=-\theta_i \rho_i\sigma^v_i \varsigma^i(s) \psi^i + (D^\top \psi)_i + \gamma \frac {(\sigma^v_i)^2} 2 (1- \rho^2_i) (\varsigma^i(s)\psi^i)^2 + \int_{E} h_\gamma(\psi^\top \eta(e))\nu_i(de).
		\label{eq:RiccatiExpTilpsi2}
	\end{align} 
	Let $\left(Y, \Lambda, U\right)$ be defined as 
	\begin{equation}\label{eq:SolbsdeExpo}
		\left\{
		\begin{array}{ccl}
			Y_t &=&  \sum_{i=1}^d \int_0^t \zeta_i V^i_s ds + \int_t^T \Big(\int_{E}h_{\gamma}\big(\psi(T-s)^\top\eta(e))\nu_0(de)+\sum_{i=1}^d \big(\zeta_i -\frac{ \theta_i^2}{2\gamma}  +  F_i(s,\psi(T-s))\big)  g^i_t(s) \Big)ds , \\[1mm]
			\Lambda_t^i &=&  \sigma^v_i\varsigma^i(t) \psi^i(T-t) \sqrt{V^i_t}, \quad U_t(e) = \psi^\top(T-t) \eta(e),\quad  i=1,\ldots,d, \quad 0 \leq t \leq T, 
		\end{array}  
		\right.
	\end{equation}
	where $g$ $=$ $(g^1,\ldots,g^d)^\T$ is given by \eqref{eq:processg} i.e. the $\R^d$-valued process \((g_t(s))_{t\leq s}\)  is defined in~\eqref{eq:processg}.  Then,  for some $p > 1$,  $\left(Y, \Lambda, U \right)$  is a  $\mathbb{S}^{p}_{\F}([0,T], \R) \times L^2_{\F}([0,T], \R^d) \times L^2_{\F}([0, T],\tilde{N})$-valued solution to the Riccati BSDEJ~\eqref{eq:bsdeExpo}
	\begin{equation}\label{eq:bsdeExpo}
		\left\{
		\begin{array}{ccl}
			dY_t &=&  \big(\frac{1}{2\gamma} \left| \lambda_t + \gamma \Sigma \Lambda_t \right|^2 -\frac{\gamma}2|\Lambda_t|^2 - \int_{E}h_{\gamma}(U_t(e))\xi(V_t, de)\big)dt + \Lambda_t^\top dW_t + \int_{E}U_t(e)\tilde{N}(dt,de). \\[1mm]
			Y_T &=&  \int_0^T \zeta V_s ds,\quad h_{\gamma}(x)=\frac{e^{\gamma x}-\gamma x-1}{\gamma},\quad \xi(V_t, de) = \nu_0(de) + \sum_{i=1}^{d} V_t^i \, \nu_i(de)
		\end{array}  
		\right.
	\end{equation}
\end{Proposition}
\medskip
\noindent{\bf Remarks:} 1. In a one–dimensional example $d=1$, 
together with $\varsigma\equiv I$,~\cite[section~$5$]{BondiLivieriPulido2024} establishes that, there exists a continuous global solution $\psi \in C(\R_+,\R_-)$ to  the Ricatti--Volterra Equation ~\eqref{eq:RiccatiExpTilpsi1}--~\eqref{eq:RiccatiExpTilpsi2}.

\smallskip
\noindent 
2. For our case studies, in Section~\ref{Sec:Num}, we take $\nu_i\equiv0$ for every $i=1,\cdots,d$. Consequently, as $\zeta_i-\frac{\theta_i^2}{2\gamma}<0$ for \(i=1,\ldots,d,\) and $K$ satisfies the Assumption~\ref{assump:kernelJumpVolterra}, \cite[Theorem A.2.]{Gnabeyeu2026a} provides the existence of a unique global continuous solution $\psi \in C([0,T],\R^d)$ to  the inhomogeneous Ricatti--Volterra Equation ~\eqref{eq:RiccatiExpTilpsi1}--~\eqref{eq:RiccatiExpTilpsi2} such that \(\psi^i(t)< 0\) for every \(t>0\) and \(i=1,\cdots,d.\)\\

\smallskip
\noindent 3. Furthermore, note as in~\cite{Gnabeyeu2026b} that, if the matrix $D$ in the drift of the volatility process is a diagonal matrix, i.e. $D=-\diag{(\lambda_1,\dots, \lambda_d)}$, for \(i=1,\ldots,d,\) by \cite[Theorem A.2.~$(c)$]{Gnabeyeu2026a}, since $\zeta_i-\frac{ \theta_i^2}{2\gamma}<0$, $\psi^i \in C([0,T],\R_-)$ is unique global solution to the following Volterra equation  (here, for short $\bar{D}_i(s):=\lambda_i-\theta_i \rho_i \sigma^v_i \varsigma^i(T-s)$).
\begin{equation}\label{eq:comp}
	\chi(t) = \int_0^t K_i(t-s)  \Big( \zeta_i-\frac{ \theta_i^2}{2\gamma} - \bar{D}_i(s) \chi(s)  + \gamma\frac {(\sigma^v_i)^2} 2  ( 1- \rho_i^2)\varsigma^i(T-s)^2\chi(s) ^2\Big)ds, \; t\leq T.
\end{equation}

\smallskip
\noindent It follows in this case (still following~\cite{Gnabeyeu2026b}) that the condition \eqref{eq:condtheta} can be made more explicit by bounding $\psi$ with respect to the vector $\theta$.
Indeed setting for \(i=1,\ldots,d,\) \(\bar{\lambda}_i:=\inf_{t\in [0,T]} \big(\lambda_i+\sigma^v_i\rho_i\theta_i\varsigma^i(t)\big)= \lambda_i+\sigma^v_i\rho_i\theta_i \|\varsigma^i\|_\infty {\bold{1}_{\rho_i \leq 0}}, \) (assuming boundedness for the function \(\varsigma\)) and assuming that \(\bar{\lambda}_i \neq 0\), by \cite[Corollary A.3.]{Gnabeyeu2026a}, we have:
\begin{equation}
	\sup_{t \in [0,T]} |\psi^i(t)| \leq \frac{| 2\gamma \zeta_i-\theta_i^2 |}{2\gamma\bar{\lambda}_i}\int_0^T f_{\bar{\lambda}_i}(s)ds = \frac{\theta_i^2 -2\gamma \zeta_i}{2\gamma\bar{\lambda}_i}(1- R_{\bar{\lambda}_i}(T))\leq \frac{\theta_i^2 -2\gamma \zeta_i}{2\gamma\bar{\lambda}_i}, \;\; i=1,\ldots,d.
\end{equation}
where $R_{\bar{\lambda}_i}$ is the \textit{ $\bar{\lambda}_i$-resolvent} associated to the real-valued kernel $K_i$ and $f_{\bar{\lambda}_i}$ its antiderivative  (see e.g.,~\cite[Section 2.2]{EGnabeyeu2025}).
Consequently, combining those component-wise estimates, we finally obtain that, a sufficient condition on $\theta$ to ensure \eqref{eq:condtheta} would be for all $i=1,\ldots,d$ 
\begin{equation}
	\theta^2_i+\big(\sigma^v_i\|\varsigma^i\|_{\infty,T}\frac{(\theta_i^2-2\gamma \zeta_i)}{2\gamma\bar{\lambda}_i}\big)^2\big(1- R_{\bar{\lambda}_i}(T)\big)^2 + e^{\frac{(\theta_i^2-2\gamma \zeta_i)^2}{4\gamma^2\bar{\lambda}^2_i} \big(1- R_{\bar{\lambda}_i}(T)\big)^2}\int_E 
\Big(1 + e^{\|\eta(e)\|^2}\Big)
\, \nu_i(de) \leq \frac{a}{a(p)}. 
\end{equation}
\noindent The main result we derive in this case, pertaining to the optimal strategy and the initial value function, is as follows:
\begin{Theorem}\label{Thm:ExpoUtilityGeneral}
	Assume there exists a unique continuous solution \(\psi\) to the inhomogeneous Ricatti-Volterra equation with L\'evy exponential jump compensator~\eqref{eq:RiccatiExpTilpsi1}-~\eqref{eq:RiccatiExpTilpsi2}
	on the interval \( [0, T] \), such that Assumption~\ref{assm:gen} is in force.
	Then, an optimal investment strategy $(\alpha_t^*)_{t\in [0,T]}$ for the Merton portfolio problem~\eqref{Expobj} is given by 
	\begin{align}
		\label{Eq:alpha_GeneralExpo*}
		\c^*_t &=  e^{-\int_{t}^{T} r(s)ds} \big( \frac{\lambda_t}{\gamma} + \Sigma \Lambda_t \big) =  \Big(\frac1\gamma  e^{-\int_{t}^{T} r(s)ds}  \big(\theta_i + \gamma \rho_i\sigma^v_i\varsigma^i(t) \psi^i(T-t) \big) \sqrt{V_t^i}   \Big)_{1 \leq i \leq d}, \quad 0 \leq t \leq T.
	\end{align}
	Moreover, 
	\begin{equation}\label{eq:boundExpo}
		\sup_{\tau \in [0, T]}\E\Big[ \exp\left(- p\gamma e^{ \int^T_\tau r(u) du} X^{\alpha^*}_\tau\right) \Big] < \infty, \text{ for some } p > 1.
	\end{equation}
	and $\alpha^*$ is admissible. The value function defined in~\eqref{Expobj} can be written as (\(\Theta_0:=-\frac{1}{\gamma}\exp\big(-\gamma e^{\int_{0}^{T} r(s)ds}x_0\big)\) for short)
	\begin{equation}\label{eq:valExpo}
		\mathcal{V}^\Xi(x_0,V_0)=\Theta_0\exp\Big( \gamma \int_0^T \big( \int_{E}h_{\gamma}\big(\psi(T-s)^\top\eta(e)\big)\nu_0(de)+\sum_{i=1}^d \big(\zeta_i -\frac{ \theta_i^2}{2\gamma}  +  F_i(s,\psi(T-s))\big)  g^i_0(s) \Big)ds \Big).
	\end{equation}
\end{Theorem}
\medskip
\noindent{\bf Remarks:} 1. It is worth noting that, if \(\eta\equiv0\), we recover the results of~\cite[section 3.2]{Gnabeyeu2026b}
where \(\Gamma:=\exp\big(\gamma Y\big)\) provided we defined \(\widetilde{\psi}:=\gamma \psi\).

2. Equation~\eqref{eq:ansatzExpo} can be reformulated as \(J^\alpha_t := -\frac{1}{\gamma}\exp(-\gamma  e^{\int_{t}^{T} r(s)ds}X^\alpha_t+ \gamma Y_t)\) and one checks with It\^o-L\'evy's Lemma together with martingale representation (see~\cite[Proof of Proposition 3.11]{Gnabeyeu2026b} that $0<\Gamma_t=\exp\big(\gamma Y_t\big)\leq1$, $\P-a.s.$, which yields that the uniform integrability of \(J^\alpha\) reduces to that of the family in~\eqref{eq:AdmStrat2} 

\medskip
\noindent We are now able to compute the indifference price for a derivative $\Xi$. The indifference buy price (bid) is the price $p^b$ at which the investor is
indifferent (in the sense that his expected utility under optimal trading is unchanged)
between paying nothing and not having the claim $\Xi$ and paying $p^b$ now to receive
the claim $\Xi$ at time $T$. That is, the investor is willing to pay at most the amount $p^b$ today for one unit of
the claim $\Xi$ at time $T$. Mathematically speaking, $p^b$  satisfy the equation 
\begin{equation*}
    \mathcal{V}^{-\Xi} ( x_0 - p^b, V_0) = \mathcal{V}^0 (x_0, V_0).
\end{equation*}
We define in a symetric way the indifference selling price.\\

\begin{Proposition}
The indifference buying price $p^b$ of one unit of $\Xi$ is explicitly given by
\begin{align*}
p^b
&= e^{-\int_0^T r(s)\,ds}
\Bigg[
\int_0^T \int_E
\Big(
h_{\gamma}\big(\psi_0(T-s)^\top \eta(e)\big)
- h_{\gamma}\big(\psi(T-s)^\top \eta(e)\big)
\Big)\,\nu_0(de)\,ds \\
&\quad
+ \sum_{i=1}^d \int_0^T
\big(
F_i(s, \psi_0(T-s)) - F_i(s, \psi(T-s)) + \zeta_i
\big)\, g_0^i(s)\, ds
\Bigg].
\end{align*}
Here, $\psi_0, \psi$ solves the Riccati--Volterra equation
\begin{align*} 
\psi_0^i(t)
&= \int_0^t K_i(t-s)
\left(
    -\frac{\theta_i^2}{2} + F_i(T-s, \psi_0(s))
\right)\, ds, \\
\psi^i(t)
&= \int_0^t K_i(t-s)
\left(
    -\zeta_i - \frac{\theta_i^2}{2} + F_i(T-s, \psi(s))
\right)\, ds,
\quad i = 1, \ldots, d.
\end{align*}
and $F_i$ defined in Equation~\ref{eq:RiccatiExpTilpsi2} of Proposition~\ref{prop:ExpoExp_riccati}.
\end{Proposition}

\noindent {\bf Proof: }
The value function $\mathcal{V}^0(x_0, V_0)$ has the form 
\begin{align*}
    \Theta_0\exp\Big( \gamma \int_0^T \big( \int_{E}h_{\gamma}\big(\psi_0(T-s)^\top\eta(e)\big)\nu_0(de)+\sum_{i=1}^d \big( -\frac{ \theta_i^2}{2\gamma}  +  F_i(s,\psi(T-s))\big)  g^i_0(s) \Big)ds \Big).
\end{align*}
Equating $\mathcal{V}^0(x_0,V_0)$ and $\mathcal{V}^{-\Xi}(x_0-p^b,V_0)$ we find $p^b$.

\subsection{Optimal strategy for the power utility maximization problem}\label{subsect:Power}
\noindent Here, with a slightly different formulation, we denote by  $\pi_t$ the proportion of wealth invested in the risky assets \(S\).
An agent invests at any time $t$ the proportion $\pi_t$ of his wealth $X^\pi_t$ in the stocks \(S\)
of price's vector $S_t$, and the remaining proportion $1-\pi_t^\top {\bold{1}_d}$ in a bond of price $S_t^0$
with interest rate $r(t)$. 
Assuming the model~\eqref{eq:stocks}--~\eqref{eq:hestonS} for $S_t$, the dynamics of the controlled wealth process is given by
\begin{align*}
	\mathrm d X^\pi_t
	&= X^\pi_t \bigl(  \pi_t^\top \big( r(t) {\bold{1}_d} + \sigma(V_t) \lambda_t  \big) + (1 - \pi_t^\top {\bold{1}_d})r(t) \bigr)\,\mathrm dt + X^\pi_t \pi_t^\top\sigma(V_t)  \,\mathrm d B_t \\
	&= X^\pi_t \bigl(r(t) + \pi_t^\top  \sigma(V_t) \lambda_t\bigr)\,\mathrm dt + X^\pi_t \pi_t^\top   \sigma(V_t) \,\mathrm d B_t.
\end{align*}
where ${\theta}=(\theta_1,\dots,\theta_d)^\top$.
\medskip
\noindent From now, we let $\alpha_t := \sigma(V_t)^\top \pi_t$ be the investment strategy. Then, under a fixed portfolio strategy, the wealth process $X^\alpha_t$, controlled by $\alpha$ is governed by:
\begin{equation}\label{Eq:wealth}
	d X^\alpha_t = \big(r(t) + \alpha_t^\top \lambda_t \big)  X^\alpha_t dt + \alpha_t^\top X^\alpha_t dB_{t}, \; X_0 = x_0 > 0.
\end{equation}
By solving the linear SDE \eqref{Eq:wealth}, the wealth process admits the explicit representation
\begin{equation}\label{eq:wealthPowerSol}
	X_t^\alpha
	=
	x_0 \exp\!\left(
	\int_0^t \Big(r(s) + \alpha_s^\top \lambda_s - \tfrac12 \left|\alpha_s\right|^2\Big)\,ds
	+
	\int_0^t \alpha_s^\top \, dB_{s}
	\right), \quad x_0\geq0.
\end{equation}
By a solution to~\eqref{Eq:wealth}, we mean an $\F$-adapted process $X^{\alpha}$ satisfying~\eqref{Eq:wealth} on $[0,T]$  with $\P$-a.s. continuous sample paths and such that 
\begin{align}
	\label{eq:estimateXPower}
	\E\big[\sup_{t\leq T} |X^{\alpha}_t|^p \big] &< \;  \infty\quad \text{for some }\quad p>1.
\end{align}

	The set of all admissible investment strategies is denoted as $\cA$ and is naturally defined by
	\begin{equation}\label{eq:AdmStrat}\mathcal A  = \left\{ \alpha \in {L^{2,loc}_{\F}([0,T], \R^d)} \mbox{ such that \eqref{Eq:wealth} has a  solution satisfying } ~\eqref{eq:estimateXPower}. 
	\right\}\end{equation}

\noindent We want to solve the Merton power utility optimization problem, i.e. our aim is to find the value function $\mathcal{V}(x_0,V_0)$ for the CRRA utility function such that
\begin{equation}\label{obj_power}
	\mathcal{V}(x_0,V_0)=	\sup_{ \alpha(\cdot) \in \cA }\E_{x_0,V_0}[U(X_T^{\alpha})]
\end{equation}
where $\E_{x_0,V_0}$ is the conditional expectation given $x_0$ and $V_0$ and $U$ the power utility function \(U(x):=\frac{1}{\gamma}x^\gamma\). The parameter $\gamma \in (0, 1)$ represents the relative
risk aversion of the investor. Smaller $\gamma$ correspond to higher risk aversion. 
A strategy $\alpha^*$ for which the supremum is attained is called an optimal strategy.

\noindent The solution method is similar to the exponential utility case. With slightly abuse of notations, we still use $\psi(\cdot)$, $Y$, etc. However, they are redefined and not mixed with counterparts in exponential utility case.

\medskip
\noindent From the dynamic of the controlled wealth process given in Equation~\eqref{eq:wealthPowerSol} 
we have
	\begin{equation}\label{eq:PowerUtil}
		U(X_T^\alpha):= \frac{1}{\gamma}(X_T^\alpha)^\gamma
		=
		\frac{1}{\gamma}(x_0)^\gamma \exp\!\left(
	\int_0^T \gamma\Big(r(s) + \alpha_s^\top \lambda_s - \tfrac12 \left|\alpha_s\right|^2\Big)\,ds
	+
	\int_0^T \gamma\alpha_s^\top \, dB_{s}
	\right).
	\end{equation}
Similar to~\cite{HuImkellerMueller2005},
	we consider the family of stochastic processes \((J^\alpha)^\alpha\)  defined for every \( t\in[0,T]\) by 
    \begin{equation}\label{eq:ansatzPower}
J_t^\alpha:=\frac{x_0^\gamma}{\gamma} \exp\!\left(
	\int_0^t\gamma \Big(r(s) + \alpha_s^\top \lambda_s - \tfrac12 \left|\alpha_s\right|^2\Big)\,ds
	+
	\int_0^t \gamma\alpha_s^\top \, dB_{s} + Y_t
	\right)
    \end{equation}where the triplet  \((Y,\Lambda, U)\) satisfies a backward stochastic differential equation with jumps (BSDEJ) under \(\P\) with a driver \(f: [0,T] \times \R \times \R^d \times \mathbb{R} \to \R\) of the form:
	\begin{equation}\label{eq:GammaDefPower}
		\left\{
		\begin{array}{ccl}
			dY_t &=&  \; -f(t, Y_t ,\Lambda_t, U_t)dt + \Lambda_t^\top dW_t + \int_{E}U_t(e)\tilde{N}(dt,de), \\
			Y_T &=&  0. 
		\end{array}  
		\right.
	\end{equation}
\noindent In these terms, we are bound to choose a function \(f\) for which \(J_t^\alpha\) is a
	supermartingale for all \(\alpha\in \mathcal A\),  there exists a \(\alpha^*\in \mathcal A\) such that \(J_t^{\alpha^*}\) is a martingale. The main result we present in this case is as follows:
\begin{Proposition}
	\label{prop:ExpoPower_riccati} 
	Assume that there exists  a solution $\psi \in C([0,T],\R^d)$ to the inhomogeneous Riccati-Volterra equation with L\'evy exponential jump compensator:
	\begin{align}
		 \psi^i(t)&= \int_0^t K_i(t-s)\big(\frac{\gamma \theta_i^2}{2(1-\gamma)}  + F_i(T-s,\psi(s))\big) ds, \quad i=1,\ldots,d,
		  \label{eq:RiccatiPower1} \\
		F_i(s,\psi) &=  \frac{\gamma}{1-\gamma} \theta_i \rho_i \sigma^v_i \varsigma^i(s) \psi^i + (D^\top \psi)_i + \frac {(\sigma^v_i)^2} 2 \left[  (\varsigma^i(s)\psi^i)^2+\frac{\gamma}{1-\gamma}\rho_i^2(\varsigma^i(s)\psi^i)^2\right] + \int_{E}h(\psi^\top \eta(e))\nu_i(de)  
		\label{eq:RiccatiPower2}
	\end{align} 
	Let $\left(Y, \Lambda, U \right)$ be defined as 
	\begin{equation}\label{eq:SolbsdePower}
		\left\{
		\begin{array}{ccl}
			Y_t &=&  \gamma\int_t^T r(s) ds +  \int_t^T \Big(\int_{E}h\big(\psi(T-s)^\top\eta(e)\big)\nu_0(de)+\sum_{i=1}^d  \big(\frac{\gamma \theta_i^2}{2(1-\gamma)}  + F_i(s,\psi(T-s))\big) g^i_t(s) \Big) ds, \\[1mm]
			\Lambda_t^i &=&  \sigma^v_i\varsigma^i(t) \psi^i(T-t) \sqrt{V^i_t}, \quad U_t(e) = \psi^\top(T-t) \eta(e),\quad  i=1,\ldots,d, \quad 0 \leq t \leq T, 
		\end{array}  
		\right.
	\end{equation}
	where $g$ $=$ $(g^1,\ldots,g^d)^\T$ is given by \eqref{eq:processg} i.e. the $\R^d$-valued process \((g_t(s))_{t\leq s}\)  is defined in~\eqref{eq:processg}.  Then, $\left(Y, \Lambda, U \right)$  is a  $\mathbb{S}^{\infty}_{\F}([0,T], \R_-) \times L^2_{\F}([0,T], \R^d) \times L^2_{\F}([0, T],\tilde{N})$-valued solution to the BSDEJ~\eqref{eq:bsdePower}
	\begin{equation}\label{eq:bsdePower}
		\left\{
		\begin{array}{ccl}
			dY_t &=&  \big(- \gamma r(t)  - \frac{\gamma}{2(1-\gamma)} \left| \lambda_t+ \Sigma \Lambda_t \right|^2 -\frac12|\Lambda_t|^2 - \int_{E}h(U_t(e))\xi(V_t, de)\big)dt + \Lambda_t^\top dW_t + \int_{E}U_t(e)\tilde{N}(dt,de). \\[1mm]
			Y_T &=&  0,\quad h(x)=e^x-x-1,\quad \xi(V_t, de) = \nu_0(de) + \sum_{i=1}^{d} V_t^i \, \nu_i(de)
		\end{array}  
		\right.
	\end{equation}
\end{Proposition}
\noindent{\bf Remark:}  For our numerical case studies, in Section~\ref{Sec:Num}, we set $\nu_i\equiv0$ for every $i=1,\cdots,d$.  Fix \(T>0\), since $K$ satisfies Assumption~\ref{assump:kernelJumpVolterra},
\cite[Theorem A.2.]{Gnabeyeu2026a} guarantees that there exists a unique non-continuable continuous solution \((\psi, T_{\max})\) to  the inhomogeneous Ricatti--Volterra Equation~\eqref{eq:RiccatiPower1}-~\eqref{eq:RiccatiPower2} with $\psi \in C([0,T_{max}),\R^d)$ in the sense that $\psi$ satisfies~\eqref{eq:RiccatiPower1}-~\eqref{eq:RiccatiPower2} on $[0,T_{max})$ with $T_{max} \in (0,T]$ and $\sup_{t<T_{max}}|\psi( t)| = +\infty$, if $T_{max}<T$. 
\begin{note}
1. The exponential and power utility portfolio maximization problems reduce to establishing existence and uniqueness of solutions to the associated BSDEs~\eqref{eq:bsdeExpo}--\eqref{eq:bsdePower}. These are in turn guaranteed by the well-posedness of the corresponding Riccati--Volterra systems~\eqref{eq:RiccatiExpTilpsi1}--\eqref{eq:RiccatiExpTilpsi2} and~\eqref{eq:RiccatiPower1}--\eqref{eq:RiccatiPower2}.

\noindent 2. The terminology “Riccati BSDE with jumps” stems from the fact that, upon removing the jump terms, these BSDEs reduce to the classical Riccati backward stochastic differential equations studied in the literature; see, for example, 
~\cite[Theorem 3.1]{ChiuWong2014}.

\noindent 3. These BSDEJs have a driver or generator which is quadratic in the Brownian stochastic integrand (denoted by \(\Lambda\) here and usually \(z\)) and with an exponential term involving the Poisson stochastic integrand variable (denoted by \(U\) here and usually by \(u\)). 

\end{note}
\noindent The main result we provide for this case for the optimal strategy and the initial value function is the following:
\begin{Theorem}\label{Thm:powerUtilityGeneral}
	Assume there exists a unique continuous solution \(\psi\) to the inhomogeneous Ricatti-Volterra equation with L\'evy exponential jump compensator~\eqref{eq:RiccatiPower1}-~\eqref{eq:RiccatiPower2}
	on the interval \( [0, T_{\text{max}}) \) so that Assumption~\ref{assm:gen} is in force. 
	Then for any $T< T_{\operatorname{max}}$, an optimal investment strategy $(\alpha_t^*)_{t\in [0,T]}$ for the Merton portfolio problem~\eqref{obj_power} is given by 
	\begin{align}
		\label{Eq:alpha_Generalpower*} 
		\alpha^*_t &= \frac{1}{1 - \gamma} \left( \lambda_t +  \Sigma \Lambda_t \right) = \;  \Big(\frac{1}{1 - \gamma}  \big(\theta_i +  \rho_i\sigma^v_i \varsigma^i(t)\psi^i(T-t) \big) \sqrt{V_t^i}   \Big)_{1 \leq i \leq d}, \quad 0 \leq t \leq T.
	\end{align}
	Moreover,
	\begin{equation}\label{eq:boundPower}
		\E\Big[ \sup_{ t \in [0, T]} |X^{\alpha^*}_t|^p \Big] < \infty,\quad \text{for some sufficiently large}\quad p>1.
	\end{equation}
	and $\alpha^*$ is admissible. The value function defined in~\eqref{obj_power} can be written as (\(\Theta_0:=\frac{x_0^\gamma}{\gamma}\exp\big( \int_0^T \gamma r(s)ds\big)\) for short)
	\begin{equation}\label{eq:valPower}
	\mathcal{V}(x_0,V_0)=\Theta_0\exp\Big( \int_0^T  \Big(\int_{E}h\big(\psi(T-s)^\top\eta(e)\big)\nu_0(de) +  \sum_{i=1}^d  \big(\frac{\gamma \theta_i^2}{2(1-\gamma)}  + F_i(s,\psi(T-s))\big) g^i_0(s) \Big) ds \Big).
	\end{equation}
\end{Theorem}


\subsection{Optimal strategy for the logarithm utility maximization problem.}\label{subsect:Log}
 In this section, we assume the model~\eqref{eq:stocks}--~\eqref{eq:hestonS} for $S_t$, that the dynamics of the controlled wealth process is given by equations~\eqref{Eq:wealth}--\eqref{eq:wealthPowerSol} and the utility function is of the form \(U(x):=\log(x)\).
As in the previous section, the investor have not a liability at maturity $T$. Trading strategies and the wealth process are defined as in the power utility case.
The set of all admissible investment strategies is denoted as $\mathcal{A}$ and is naturally defined by
\begin{equation}\label{eq:AdmStrat}
    \mathcal{A} = \left\{ (\alpha_t)_{t \in [0,T]} \in {L^{2}_{\F}([0,T], \R^d)} 
    \right\}. 
\end{equation} 

 The optimization utility problem is given by 
\begin{equation}\label{obj_log}
\begin{aligned}
    \mathcal{V}(x_0, V_0) 
    &= \sup_{\alpha(\cdot) \in \mathcal{A}} 
    \mathbb{E}_{x_0,V_0}\bigl[\log(X_T^{\alpha})\bigr] ,\qquad X_0^{\alpha}:=x_0>0\\
    &= \log(x_0) + \sup_{\alpha(\cdot) \in \mathcal{A}} 
    \mathbb{E}\biggl[\int_0^T \Bigl(r(s) + \alpha_s^\top \lambda_s 
    - \tfrac{1}{2}\left|\alpha_s\right|^2\Bigr)ds
    + \int_0^T \alpha_s^\top\,dB_s\biggr].
\end{aligned}
\end{equation}
As in the previous section, we want to determine a family of process $J^\alpha$ satisfying $J^\alpha_T = \log(X_T^\alpha)$ with an initial value that does not depend on $\alpha$. Moreover \(J_t^\alpha\) is a
supermartingale for all \(\alpha\in \mathcal A\) and there exists a \(\alpha^*\in \mathcal A\) such that \(J_t^{\alpha^*}\) is a martingale. The strategy $\alpha^*$ is an optimal strategy and $J^{\alpha^*}_0$ is the value function of the optimization problem \eqref{obj_log}. We can choose the family of stochastic processes $(J^\alpha)_{\alpha}$, 
defined for every $t \in [0,T]$ by
\begin{equation}\label{eq:ansatzlog}
    J_t^\alpha := \log(x_0) + \int_0^t \Bigl(r(s) + \alpha_s^\top \lambda_s 
    - \tfrac{1}{2}\left|\alpha_s\right|^2\Bigr)ds
    + \int_0^t \alpha_s^\top\,dB_s + Y_t,
\end{equation}
where the triplet $(Y, \Lambda, U)$ satisfies the following backward SDE with jumps (BSDEJ) on $[0,T]$ under $\mathbb{P}$:
\begin{equation}\label{eq:BSDElog}
    \left\{
    \begin{array}{rcl}
        dY_t &=& -\big( r(t) + \frac{1}{2}|\lambda_t|^2\big)\,dt + \Lambda_t^\top\,dW_t 
                 + \int_{E} U_t(e)\,\tilde{N}(dt,de), \\
        Y_T  &=& 0,
    \end{array}
    \right.
\end{equation}
For $t \in [0,T]$, the driver of the BSDEJ \eqref{eq:BSDElog} is $\mathcal{F}_t$-measurable. Classical results on BSDEs with jumps (see for example \cite{oksendal2019applied}) yield the existence and 
uniqueness of a solution $ (Y, \Lambda, U) \in \mathbb{S}^{2}_{\mathbb{F}}([0,T],\mathbb{R}) 
    \times L^2_{\mathbb{F}}([0,T], \mathbb{R}^d) 
    \times L^2_\F([0,T],\tilde{N})$
to \eqref{eq:BSDElog}.

\begin{Theorem}\label{Thm:logUtilityGeneral}
Let \(T>0\) be fixed.
	An optimal investment strategy $(\alpha_t^*)_{t\in [0,T]}$ for the Merton portfolio problem~\eqref{obj_log} is given by 
	\begin{align}
		\label{Eq:alpha_GeneralLog*}
		\c^*_t &= \lambda_t = \;  \Big(\theta_i  \sqrt{V_t^i}   \Big)_{1 \leq i \leq d}, \quad 0 \leq t \leq T.
	\end{align}
	Moreover $\alpha^*$ is admissible and the value function or maximum expected utility defined in~\eqref{obj_log} can be written as
\begin{equation}\label{eq:valueLog} 
    \sup_{\alpha \in \mathcal{A}}
    \mathbb{E}_{x_0,V_0}
    \Bigl[\log\bigl(X_T^{\alpha^*}\bigr)\Bigr]
    = \mathcal{V}(x_0,V_0)
    = \log(x_0) + \int_0^T r(s) ds +  
    \frac{(\theta\odot\theta)^\top}{2} \int_0^T \big(I_d - \int_0^{T-s} R_D(u)  du\big) g_0(s) ds.
\end{equation}
where $\odot$ denote the Hadamard (pointwise or component-wise) product and $R_D$ is the  \textit{resolvent of the second kind associated with the matrix-valued kernel $-K D$}.
\end{Theorem}

\section{Numerical experiments: The case of a two-dimensional rough Heston volatility with jumps}\label{Sec:Num}
	\noindent In this section, we illustrate the results of Section~\ref{sect-MertProblem} by numerically computing the optimal portfolio strategy for a special case of two-dimensional rough Heston model with jump as described in   sections~\ref{sect-prelim} with $\varsigma\equiv I$. We consider a financial market consisting of one risk-free asset and \(d = 2\) risky assets, with an investment horizon of \(T = 1\) year. 	
To model the roughness of the asset price dynamics, we employ an appropriate integration kernel. We choose a fractional kernel of Remark~\ref{rm:Kernels} of the form:
\[
K(t) = \begin{pmatrix}
	\frac{t^{\alpha_1-1}}{\Gamma(\alpha_1)} & 0 \\
	0 & \frac{t^{\alpha_2-1}}{\Gamma(\alpha_2)}
\end{pmatrix}, \quad 0.1 +\frac12=\alpha_1,\, 0.4 +\frac12=\alpha_2 \in \big( \frac{1}{2}, 1 \big).
\]
Here, \( \Gamma(.) \) is the Gamma function, and the parameter \( \alpha \) controls the degree of roughness in the model. 
For the jumps part, we consider $E = \mathbb{R}$, $\eta^1(e) = \eta^2(e) = \eta(e) = \kappa e \mathbf{1}_{\{e > 0\}}$ for some $\kappa >0$ to control variance of the jump size, and let $N(dt, de)$ be a Poisson random measure with compensator $\nu_0(de) = \beta \mathcal{N}(0,1)(de)$, where $\beta > 0$. We set $\nu_i\equiv0$ for every $i=1,\cdots,d$. 
Note that, the model is sufficiently rich to capture several well-known stylized facts of financial markets:

\begin{itemize}
	\item Each asset \(S^i\), \(i=1,2\) exhibits stochastic rough volatility with jumps driven by the process \(V^i\), with different Hurst indices \(\alpha_i\) and an intensity $\beta$.
	We can even assume a corellation between \(B^1\) and \(B^2\) through \(\rho \in (-1,1)\).
	\item Each stock \(S^i\) is correlated with its own volatility process through the parameter \(\rho_i\) to take into
	account the leverage effect.
\end{itemize}
We consider the setting where in Equation~\eqref{VolSqrt_}, the simplified specification
\(\varphi(t) = I_{2\times2}, \; t \ge 0, \) holds almost surely, in which case the \(\R^d-\) valued mean-reverting function \(\mu\) is constant in time, that is, \(\mu(t) = \mu_0 \in \R^2, \; \mbox{for all }t \ge 0\).
	\medskip
	\noindent We consider the following estimates for the model parameters:
	\[
	V_0 = \begin{pmatrix}
		0.01 \\
		0.03
	\end{pmatrix}, \quad
	\mu_0 = \begin{pmatrix}
		2.0 \\
		2.5
	\end{pmatrix},\;
	D = \begin{pmatrix}
		-0.2 & 0 \\
		0 & -0.6
	\end{pmatrix}, \quad
	\Sigma = \begin{pmatrix}
		-0.7 & 0 \\
		0 & -0.55
	\end{pmatrix}, \;
	\theta = \begin{pmatrix}
		0.1 \\
		0.1 
	\end{pmatrix},\;
	\sigma^v = \begin{pmatrix}
	0.4 \\
	0.2
	\end{pmatrix}.
	\]
\medskip
\noindent {\bf Remark:}
	In order to numerically implement the optimal strategy \eqref{Eq:alpha_Generalpower*}--\eqref{Eq:alpha_GeneralExpo*}, one needs to simulate  the non-Markovian process $V$ in Equation~\eqref{VolSqrt_}--~\eqref{VolSqrt2} and to discretize the Riccati-Volterra equation  for $\psi$ in~\eqref{eq:RiccatiPower1}--~\eqref{eq:RiccatiPower2} and~\eqref{eq:RiccatiExpTilpsi1}--~\eqref{eq:RiccatiExpTilpsi2} respectively.\\

\subsection{About the numerical scheme for the Volterra process with jumps and the fractional Riccati equations}
\noindent To simulate the volterra process, we use the $K$-integrated discrete time Euler-Maruyama scheme defined below on the time grid $t_k =t^n_k =\frac{kT}{n}, k=0, \dots, n$ and in the fractional kernel case, namely for \(i=1,2\), $\;\overline V^{i,n}_{0}=V^{i,0}$ and for every $k=1,\ldots,n$,  
\begin{align*}\label{eq:EulerXdisc}
	\overline V^{i,n}_{t_{k}} 
	&= V^{i,0} +  \sum_{\ell=1}^{k} \big(\mu_0^i -\lambda_i\overline{V}_{t_{\ell-1}}^{i,n} - \int_E \eta^i(e) \nu_0(de)\big) \int_{t_{\ell-1}}^{t_{\ell}} K_{\alpha_i}(t_{k} -s)ds   \\
    &+  \sum_{\ell=1}^{k} \sigma^v_i\sqrt{ \overline{V}_{t_{\ell-1}}^{i,n}} \int_{t_{\ell-1}}^{t_{\ell}} K_{\alpha_i}(t_{k} -s) \, dW_s + \sum_{\ell=1}^k \int_{t_{\ell-1}}^{t_{\ell}} \int_E K_{\alpha_i}(t_k-s)\eta^i(e) N(ds, de).
\end{align*}

\smallskip
\noindent In the discretization procedure, both deterministic and stochastic integrals must be handled. To this end, for each \(i=1,2\), let us denote by $C^i=(C_{k\ell}^i)_{k,\ell=1:n}$ the $n\times n$ lower triangular matrix involving the deterministic integrals, by $G^i= (G_{k\ell}^i)_{ k=1:n+1, \ell=1: n}$ the  $(n+1)\times n$ matrix  involving the random terms $\int_{t_{\ell-1}}^{t_{\ell}}K_{\alpha_i}(t_{k}-s)dW_s^i$ and $P^i = (P^i_{kl})_{k=1:n+1, \ell = 1:n}$ the $n\times n$ lower triangular matrix involving the jumps part $\int_{t_{\ell-1}}^{t_{\ell}} \int_E K_{\alpha_i}(t_k-s)\eta^i(e) N(ds, de)$.
\begin{align*}
&(C_{k\ell}^i)_{k,\ell=1:n} 
= \left(
\int_{t_{\ell-1}}^{t_\ell} K_{\alpha_i}(t_k - s)\,ds \,\mathbf{1}_{\{1 \le \ell \le k \le n\}}
\right), \,(G_{k\ell}^i)_{ k=1:n+1}^{\ell=1: n}= \left(
\begin{array}{l}
\int_{t_{\ell-1}}^{t_\ell} K_{\alpha_i}(t_k - s)\,dW_s^i \,\mathbf{1}_{\{1 \le \ell \le k \le n\}} \\[0.3em]
\Delta W_{t_\ell}^i, \quad k = n+1,\; \ell=1:n
\end{array}
\right), \\
&\hspace{2cm}(P_{k\ell}^i)_{k,\ell=1:n} 
= \left(\sum_{m: t_{\ell-1}\leq T_m\leq t_{\ell}} K_{\alpha_i}(t_k - T_m)\eta^i(e_m) \,\mathbf{1}_{\{1 \le \ell \le k \le n\}} 
\right).
\end{align*}
with $(T_m)$ a sequence of arrival times such that $(T_m - T_{m-1})$ are i.i.d with exponential law of parameter $\beta$.
The last line in \(G^i\) has been introduced in order to be able to perform a consistent joint simulation of $\bar V^{n}$ and  the Euler-Maruyama scheme of the Markovian 
wealth process $X$ in ~\eqref{eq:wealth} and~\eqref{Eq:wealth} depending on $\bar V^{n}$, $W$ and $N$ (see e.g. \cite{Gnabeyeu2026a} for further details). 

\smallskip
\noindent 
Then the following relation holds for \(i=1,2\): $\;\overline V^{i,n}_{0}=V^{i,0}$ and for every $k=1,\dots,n$
\begin{equation}\label{eq:VectSimul}
\big( \overline V^{i,n}_{t_{k}} \big)_{k=1:n}= V^{i,0}\mbox{\bf 1}_n +C_{k,\cdot}^i\big(\mu_0^i -\lambda_i\overline{V}_{t_{j-1}}^{i,n}- \int_\R \eta^i(e) \nu_0(de)\big)_{j=1:n} +G_{k,\cdot}^i\big( \sigma^v_i\sqrt{ \overline{V}_{t_{j-1}}^{i,n}})_{j=1:n} + P^i_{k,.}\cdot \mbox{\bf 1}_n
\end{equation}
\noindent which is simulated on the discrete  grid \((t^n_k)_{0\leq k\leq n}\) by generating an independent sequence of gaussian vectors \( G^i_{k,\cdot}, k=1 \cdots n\) using an extended and numerically stable Cholesky decomposition of a well-defined covariance matrix \(\bar C\).

\noindent Further details on the simulation of the Gaussian stochastic integrals terms within this semi-integrated Euler scheme introduced in this context for Equation~\eqref{VolSqrt_}--~\eqref{VolSqrt2} can be found in~\cite[Appendix A]{EGnabeyeu2025,GnabeyeuPages2026}. 
Convergence of the scheme, along with explicit convergence rates are beyond the scope of the present paper, but can be established by adapting the results of~\cite{Bonesini2023Volterra} to incorporate jump terms. 
 Extensions to more general kernels and path-dependent coefficients can be obtained similarly, following the approach developed in~\cite{GnabeyeuPages2026}.

\medskip
\noindent 
To numerically solve the two-dimensional Riccati--Volterra system 
\eqref{eq:RiccatiExpTilpsi1}--\eqref{eq:RiccatiExpTilpsi2} and \eqref{eq:RiccatiPower1}--\eqref{eq:RiccatiPower2},  
we employ, as in~\cite{Gnabeyeu2026a, Gnabeyeu2026b} (see also~\cite[Section~5.1]{el2019characteristic}), 
the generalized Adams--Bashforth--Moulton scheme, also known as the \textit{fractional Adams method}. 
This predictor--corrector method, introduced in~\cite{DiethelmFordFreed2002,DiethelmFordFreed2004} 
for fractional ordinary differential equations, comes with proven convergence guarantees 
established in~\cite{LiTao2009}.

\noindent Over a regular or uniform discrete time-grid \((t_k)_{k = 1,\ldots,n}\) with mesh or step length \(\Delta=\frac{T}{n}\) \((t_k = k\Delta)\) for some integer \(n\), let \(y_j^i:=\psi^i(t_j)\), the explicit numerical scheme to estimate \(\psi\) is given by
\[
\begin{cases}
	y_{k+1}^{i,P}
	= \displaystyle \sum_{j=0}^{k} b^i_{j,k+1} f_i(t_j,y_j^i), \;\text{with} \;f_i(t_j,x):=a_i+ F_i(T-t_j,x)
	\\
	y_{k+1}^i
	=
	\displaystyle \sum_{j=0}^{k} a^i_{j,k+1} f_i(t_j,y_j^i)
	+ a^i_{k+1,k+1} f_i(t_{k+1},y_{k+1}^{i,P})
	, \quad y_{0}^i=0
\end{cases},\; i\in\{1,2\}. 
\]
where  \(a_i=\zeta_i-\frac{\theta_i^2}{2\gamma}\) in section~\ref{subsect:Expo} (resp. \(a_i=\frac{\gamma\theta_i^2}{2(1-\gamma)} \) in section~\ref{subsect:Power}), \(F_i\) in~\eqref{eq:RiccatiExpTilpsi2} (resp. in~\eqref{eq:RiccatiPower2}) and the weights \(a^i_{j,k+1}\), \(b^i_{j,k+1}\) are defined as
\[
a^i_{j,k+1} =
\frac{\Delta^{\alpha_i}}{\Gamma(\alpha_i+2)}
\begin{cases}
	k^{\alpha_i+1} - (k-\alpha_i)(k+1)^{\alpha_i}, & \text{if } j=0,\\
	(k-j+2)^{\alpha_i+1} + (k-j)^{\alpha_i+1}
	-2(k-j+1)^{\alpha_i+1}, & 1 \le j \le k,\\
	1, & j = k+1,
\end{cases}
\]
and
\[
b^i_{j,k+1}
= \frac{\Delta^{\alpha_i}}{\Gamma(\alpha_i+1)}
\left((k+1-j)^{\alpha_i} - (k-j)^{\alpha_i}\right),
\qquad 0\leq j\leq k.
\]
Here, \(T\) denotes the terminal time or time horizon, \(n\) the number of time steps, and \(\Delta := \frac{T}{N}\) the time increment.
\noindent Theoretical guarantees for the convergence of this numerical algorithm (the 2-dimensional fractional Adams--Bashforth--Moulton method) are established in~\cite{LiTao2009}.

\subsection{Numerical illustrations}
Assumptions~\ref{assm:gen} are satisfied under the parameter settings specified in this section’s preamble.
    \begin{figure}[H]
	\centering
	\subfloat[$H = 0.1$ and $V^1_0 = 4$]{
		\label{efficient_frontier_small_T}
		\includegraphics[width=70mm]{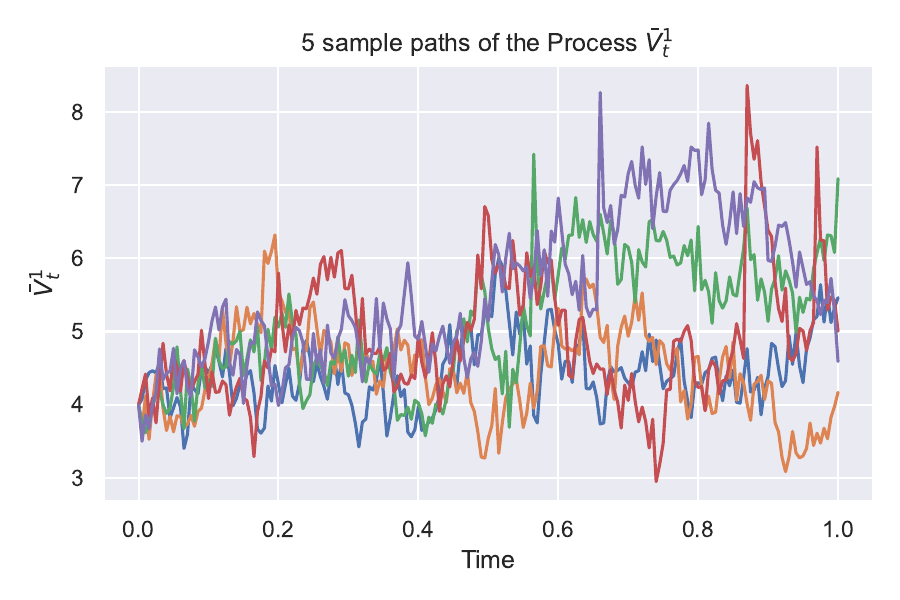}} 
	\subfloat[$H = 0.3$ and $V^2_0 = 4$]{
		\label{efficient_frontier_medium_T}
		\includegraphics[width=70mm]{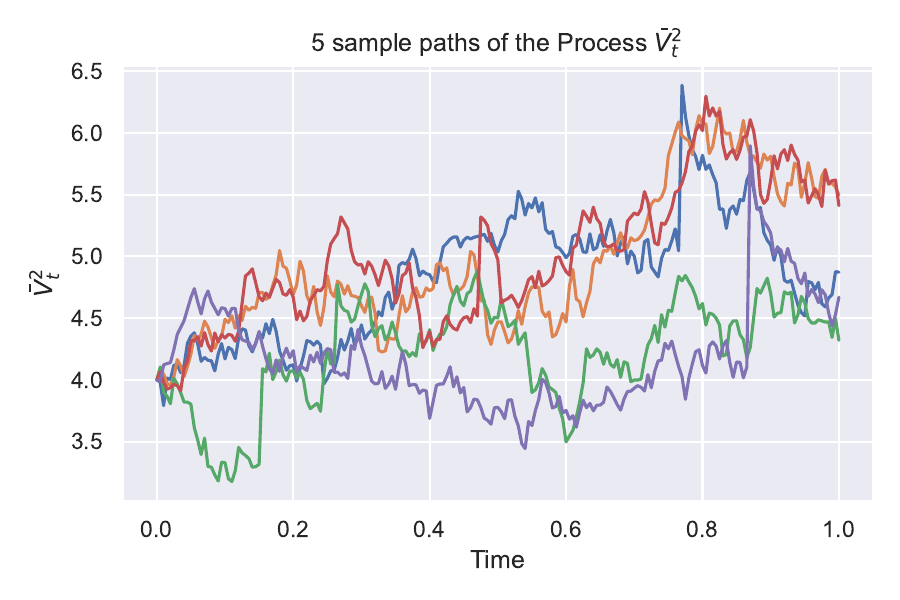} 
	}
	\caption[Variance process.]{\textit{Graph of \(5\) samples paths of the processes  \( t_k \mapsto V^1_{t_k} \) (left) and  \( t_k \mapsto V^2_{t_k} \) (right) over the time interval \( [0, 1] \), for the number of time steps \( n = 200 \).}}
	\label{fig:EfficientFront}
\end{figure}
\smallskip
\noindent Figures \ref{fig:EfficientFront} displays simulated paths of the instantaneous volatility
$V^i_{t_k}$ for different values of the Hurst index $H_i = \alpha_i - \frac{1}{2} \in (0, \frac{1}{2})$. Each upward spike corresponds to a jump arrival time $\tau_k$ of the Poisson measure $N$, at which the volatility is pushed upward by the kernel-weighted mark contribution $K_{\alpha_i}(t - \tau_k)\,\eta(e_k)$. Since $\eta(e) = \kappa e\,\mathbf{1}_{e > 0}$, only positive marks contribute, so jumps act exclusively as upward shocks on $V^i$. 
$H_i$ produces both rougher between-jump paths and slower post-spike decay. This asymmetric behaviour i.e. instantaneous upward shocks followed by gradual relaxation, is consistent with the empirical observation that volatility spikes are almost exclusively upward \cite{whaley2000investor}, reflecting the asymmetric nature of market fear: uncertainty materialises abruptly but dissipates slowly.

\begin{figure}[H]
	\centering
	\includegraphics[width=0.87\linewidth]{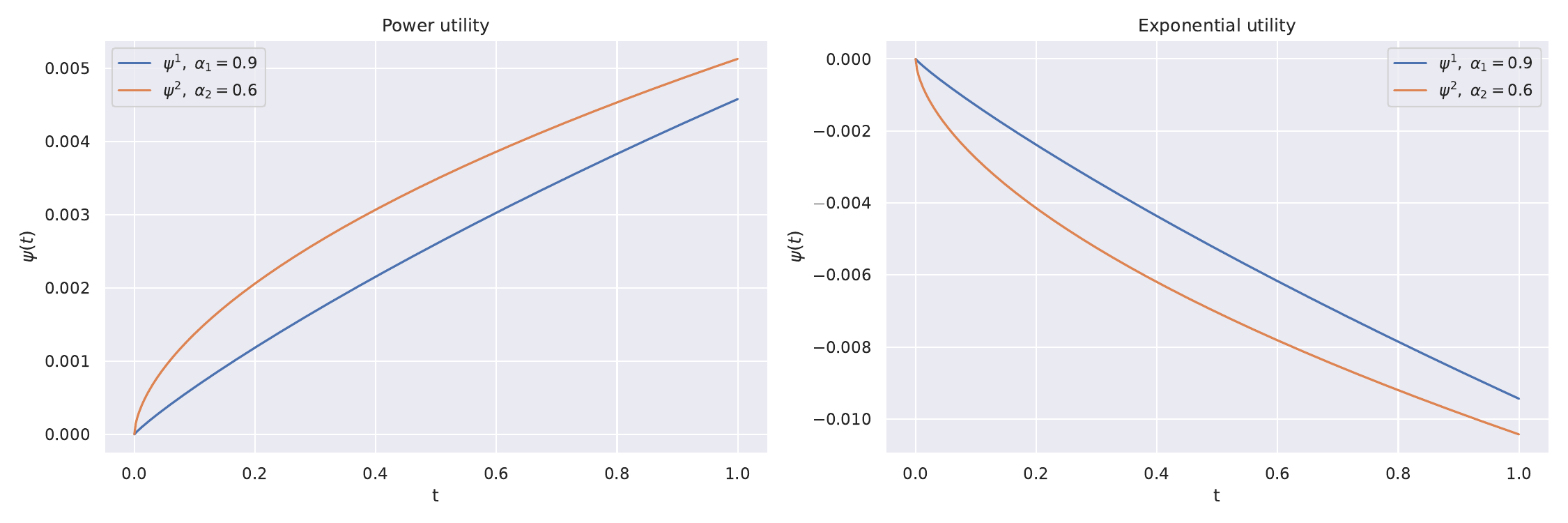}
	\caption{\textit{Graph of \( t_k \mapsto \psi^1_{t_k} \) and \( t_k \mapsto \psi^2_{t_k} \) over \( [0, 1] \) with the fractional Adams algorithm, \(\gamma=0.2\) and the number of time steps \( n = 200 \) both for Power (left) and Exponential (right) utilities functions.}}\label{Fig:Riccati}
\end{figure}

\smallskip
\noindent Figures~\ref{Fig:Riccati} above also confirm the Remark following Proposition~\ref{prop:ExpoExp_riccati}, that is the claim that $\psi \leq 0$ for every \(\rho_i \in (0,1)\) and every \(i\in\{1,2\}\) in the exponential utility case. The lower the Hurst exponent $H$, the more negative $\psi$ becomes.  In the power utility case however, the left panel of Figure~\ref{Fig:Riccati} suggests that, in accordance with~\cite{HanWong2020b,Gnabeyeu2026b}, $\psi(t) $
is positive for $t > 0 $. Moreover, its value becomes larger for a smaller hurst coefficient $H$. 

\smallskip
\noindent Since the optimal strategies $(\alpha^*_t)_{t\in[0,T]}$ given by~\eqref{Eq:alpha_Generalpower*}--\eqref{Eq:alpha_GeneralExpo*} are stochastic processes, we rather consider in Figure~\ref{fig:strategy_gamma}, the evolution of the optimal vector of amounts invested in each stock, 
that is, the associated deterministic mapping \(t \mapsto \pi_t^*,\)
(recall that $\alpha_t^* = \sigma(V_t)^{\top}\pi_t^*$ with $\sigma(V_t) = \sqrt{\mathrm{diag}(V_t)}$, 
and $\alpha^*$ is given by~\eqref{Eq:alpha_Generalpower*}--\eqref{Eq:alpha_GeneralExpo*}). 
\begin{figure}[H]
	\centering
	\begin{subfigure}{0.88\linewidth}
		\centering
		\caption*{(a) $\gamma=0.2$ and $T=1$}
		\includegraphics[width=\linewidth]{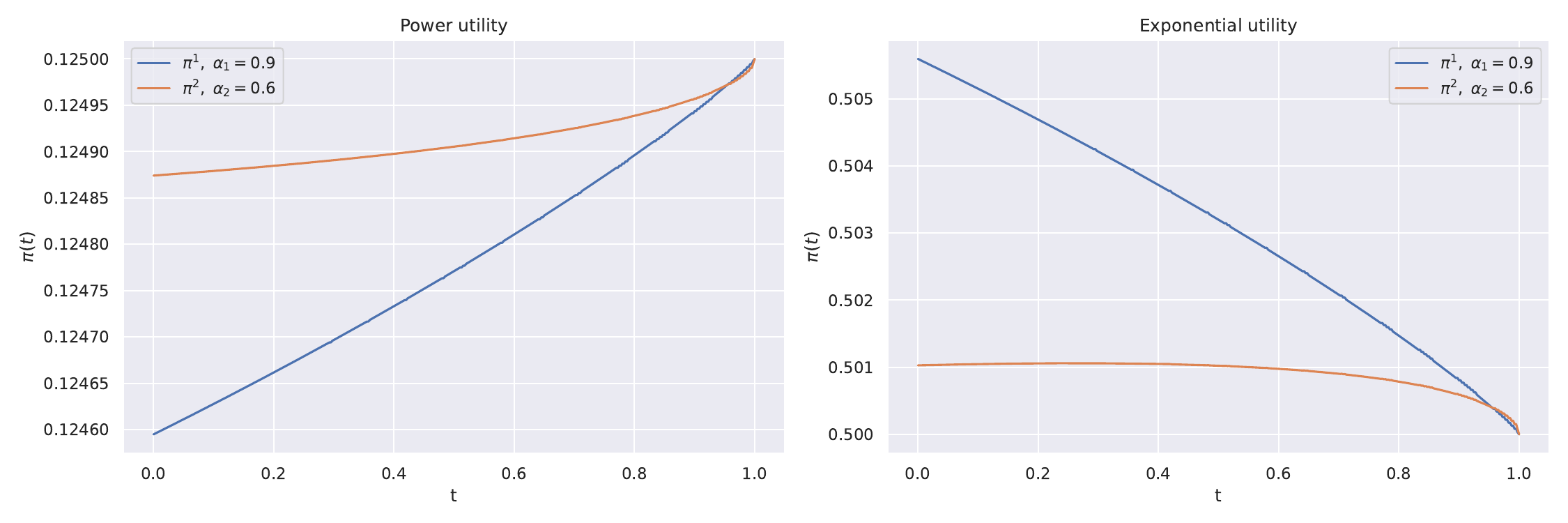}
	\end{subfigure}
	\begin{subfigure}{0.9\linewidth}
		\centering
		\caption*{(b) $\gamma=0.6$ and $T=5$}
		\includegraphics[width=\linewidth]{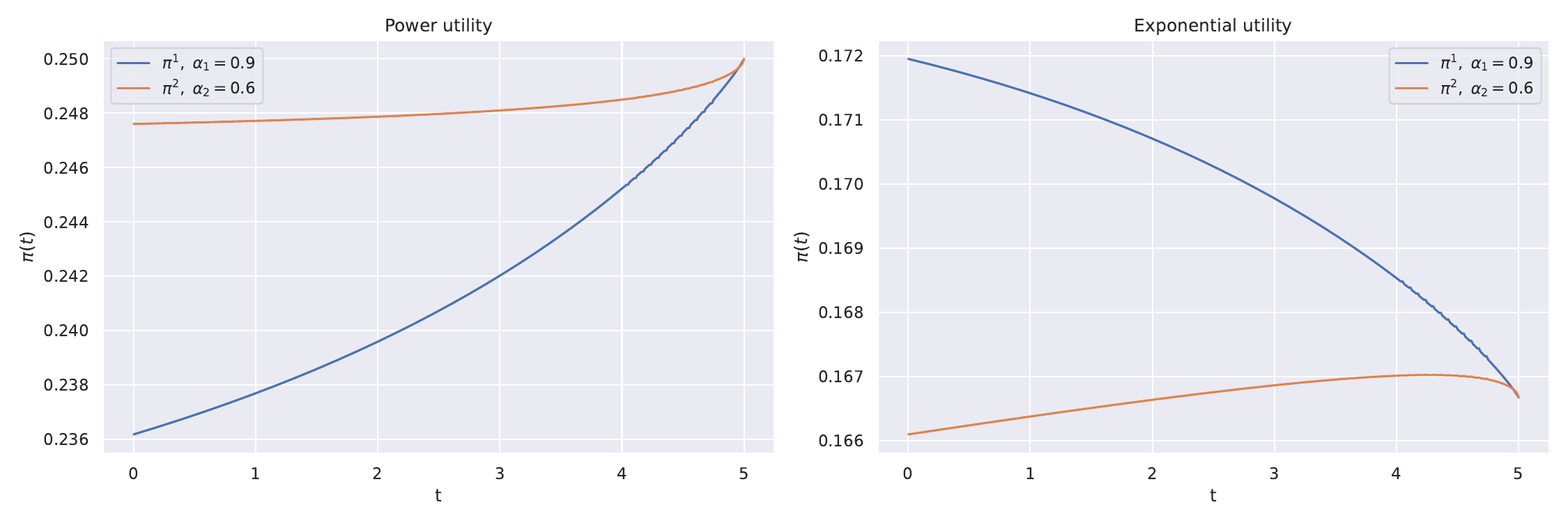}
	\end{subfigure}
	\begin{subfigure}{0.88\linewidth}
		\centering
		\caption*{(c) $\gamma=0.8$ and $T=10$}
		\includegraphics[width=\linewidth]{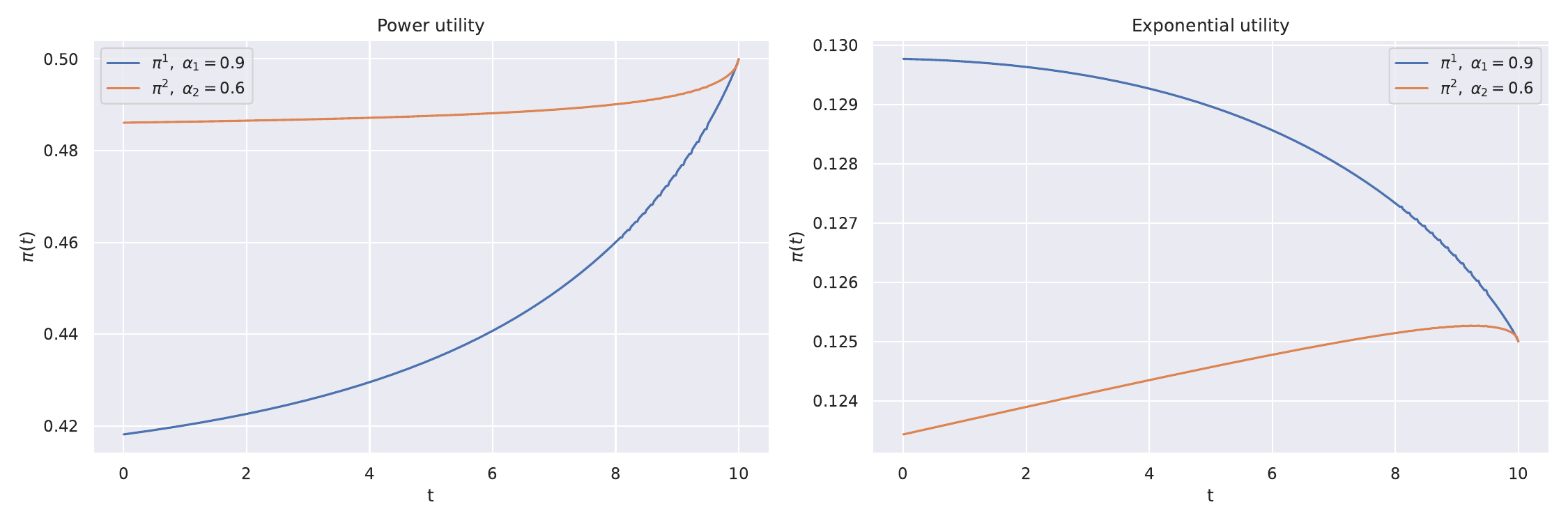}
	\end{subfigure}
	\caption{\textit{Evolution of the optimal portfolio strategy for different levels of the risk aversion parameter $\gamma$ with \(T\in\{1,5, 10\}\) for both the Power (left) and Exponential (right) utilities functions, \( r = 0.02 \).}}
	\label{fig:strategy_gamma}
\end{figure}
\smallskip


\medskip
\noindent Figure~\ref{fig:strategy_gamma} exhibits that if the stock volatility is rougher, the investment demand under power utility is higher. Interestingly,
exponential utility implies complete different opinion about roughness. It suggests investing less under
smaller $H$ (this is to be compared to~\cite{HanWong2020b,AichingerDesmettre2021, Gnabeyeu2026b}.

\smallskip
\noindent Economically, as mentionned in~\cite{HanWong2020b}, these figures show that different forms of risk aversion described by concave utility functions exhibit different
sensitivity in investment demand with respect to the volatility roughness.
More broadly, compared with~\cite{HanWong2020b,AichingerDesmettre2021,Gnabeyeu2026b}, the time-dependent multivariate Volterra integral equations characterising the optimal strategy, together with the optimal value function, are formulated in terms of the scaled compensated exponential function~\eqref{eq:ScaledExpCom_funct}, which naturally arises in the representation of jump cumulants for L\'evy processes.

\section{Proofs of the main results}\label{sect:proofMresult}
\noindent The proofs in this section are an adaptation of the proofs of~\cite[Theorem 3.7 and Theorem 3.12]{Gnabeyeu2026b}.

\smallskip
\noindent {$\rhd$ {\em Preliminaries}. As a first preliminary, we state the following Lemma which is an extension to the jump setting~\eqref{VolSqrt_} of the result from \cite[Lemma~5.1]{Gnabeyeu2026b}. 

	\begin{Lemma}[Martingale property of stochastic exponentials]\label{lm: extended_m_EG_lemma}
		Let $g_1$ and $g_2$ be two deterministic \(\R-\)valued bounded processes such that $g_1,g_2\in L^{\infty}([0,T],\R)$.
		Let $\left(\lambda, \Lambda\right) \in {L^{2,loc}_{\F}([0,T], \R^d)\times L^{2,loc}_{\F}([0,T], \R^d)}$ be two \(\R^d-\)valued stochastic processes, defined in~\eqref{eq:stocks} and~\eqref{eq:LambdaU} respectively. 
		Let $B$, $W$ be two $d$-dimensional Wiener processes and \(\Sigma\) their corellation matrix,
		then under assumption~\ref{assm:gen}, the local martingale
	\begin{align*} Z_t := \mathcal E\Big(\int_0^tg_1(s) (\lambda_s + \Sigma\Lambda_s)^\T dB_s + g_2(s) \Lambda_s^\top dW_s +  \int_0^t\int_{E}\big(e^{ U_s(e)}-1\big)\widetilde{N}(ds,de)\Big) 
	\end{align*} is a true martingale. 
	\end{Lemma}  
	\noindent \emph{Proof:} We apply a generalized version of the Novikov criterion for Lévy processes (see Theorem~8 in~\cite{protter2008no}).
Define
\begin{equation}
M_t := \int_0^t g_1(s)\big(\lambda_s + \Sigma \Lambda_s\big)^\top dB_s
+ \int_0^t g_2(s)\Lambda_s^\top dW_s
+ \int_0^t \int_E \big(e^{U_s(e)} - 1\big)\,\widetilde{N}(ds,de),
\end{equation}
and denote by $M^c$ its continuous part. Then $Z = \mathcal{E}(M)$ is a positive local martingale. It is a true martingale if the Novikov criterion holds, namely,
\begin{align}
\mathbb{E}\!\left[
\exp\!\left(
\frac{1}{2}\langle M^c \rangle_T 
+ \int_0^T \int_E \Big( (U_s(e)-1)e^{U_s(e)} + 1 \Big)\,\xi(V_s,de)\,ds
\right)
\right] < \infty.
\end{align}

By the Cauchy--Schwarz inequality, it is sufficient to verify that
\begin{align}
\mathbb{E}\!\left[\exp\!\left(\langle M^c \rangle_T\right)\right] &< \infty, \\
\mathbb{E}\!\left[\exp\!\left( 
2 \int_0^T \int_E \Big( (U_s(e)-1)e^{U_s(e)} + 1 \Big)\,\xi(V_s,de)\,ds
\right)\right] &< \infty.\label{CS_checking}
\end{align}
	We compute the quadratic variation of $M^c$. Using
	$d\langle B,W\rangle_t=\Sigma dt$, we obtain
	\begin{align*}
		\langle M^c\rangle_t
		&=
		\int_0^t g_1(s)^2|\lambda_s+\Sigma\Lambda_s|^2 ds
		+\int_0^t g_2(s)^2|\Lambda_s|^2 ds 
		+2\int_0^t g_1(s)g_2(s)(\lambda_s+\Sigma\Lambda_s)^\top \Sigma\Lambda_s ds\\
        &\hspace{3cm}
\\
		&\leq   4 C \big(1+|\Sigma|^2\big)\int_0^t \left(|\lambda_s|^2+|\Lambda_s|^2\right)ds 
	\end{align*}
where we use the elementary arithmetic mean-geometric mean (AM-GM) inequality $|ab|\leq (|a|^2+|b|^2)/2 $ and set \(C:=\max (1, \|g_1\|_{\infty,T}, \|g_2\|_{\infty,T})\geq 1\) owing to the boundedness of $g_1,g_2$.
The term $$\E \Big[ \exp\Big({4C \big(1+|\Sigma|^2\big)\int_0^T \Big( |\lambda_s|^2 + |\Lambda_s|^2\Big)  ds}\Big)    \Big]$$ if finite 
thanks to condition~\eqref{eq:assumption_novikov} (since  Assumption~\ref{assm:gen} is in force) with constant \(a(4C) =4C \big(1+|\Sigma|^2\big)\).  It remains to verify the second inequality\eqref{CS_checking}.
For this jump part, using the inequality
\[
(x - 1)e^x \leq e^{2|x|} , \qquad \text{for all } x \in \mathbb{R},
\] combine with Cauchy-Schwarz again it is sufficient to show that 
\begin{align}
\mathbb{E}\!\left[
\exp\!\left(4
\int_0^T \int_E 
\Big(1 + \exp\!\big(2 \sup_{t \in [0,T]} |U_t(e)|\big)\Big)
\, \nu_0(de)\, ds
\right)
\right] 
&< \infty, \\[0.5em]
\mathbb{E}\!\left[
\exp\!\left(
4\int_0^T \sum_{k=1}^d 
\left(
\int_E 
\Big(1 + \exp\!\big(2 \sup_{t \in [0,T]} |U_t(e)|\big)\Big)
\, \nu_k(de)
\right)
V_s^k \, ds
\right)
\right] 
&< \infty.
\end{align}
For the first inequality, since $\psi$ and $\eta$ are deterministic functions, it suffices to ensure the finiteness of the argument inside the exponential. This is always satisfied due to the finiteness of $\nu_0$ together with condition~\eqref{eq:cond_U} of Assumption~\eqref{assm:gen}. 

The second finiteness condition is then equivalent to the finiteness of the Laplace transform of $\beta^\top V$ with \(\beta:=(\beta_1, \ldots,\beta_d)^\top\) given by~\eqref{eq:beta_i}, which holds true thanks to condition~\eqref{eq:assumption_novikov} (since  Assumption~\ref{assm:gen} is in force) with constant \(a(1) =4\).
This complete the proof. \hfill $\Box$
\begin{Remark} In reference to the first claim in remark~\eqref{rm:ass_} on assumption~\ref{assm:gen}
when $U$ is non-positive (equivalently, when $\psi$ is non-positive), condition \eqref{CS_checking} is easier to satisfy. Indeed, for $x\leq0$, one has the bounds $0 \leq (x-1)e^x + 1 \leq |x|, \; x\leq 0.$ Consequently, for every $i=0,\ldots,d$ and $s\in[0,T]$,
\[0 \leq \int_E \bigl(U_s(e)-1)e^{U_s(e)} + 1\bigr)\nu_i(de) \leq \int_E |U_s(e)|\nu_i(de)\leq |\psi(T-s)|\sqrt{\nu_i(E)}
\left(\int_E |\eta(e)|^2\,\nu_i(de)\right)^{1/2}
<\infty.\]
where we have used the Cauchy--Schwarz inequality twice and the fact that $\nu_i(E)<\infty$ by the finiteness of $\nu_i$ and $\int_E\|\eta(e)\|^2\,\nu_i(de)<\infty$ by assumption~\ref{assump:kernelJumpVolterra}~$(ii)$. Moreover, since $\psi$ is continuous on the compact interval $[0,T]$, it is bounded, and hence
\[
\beta_i:=\sup_{t\in[0,T]}\int_E|U_t(e)|\,\nu_i(de)<\infty.
\]
Consequently, letting \(\beta:=(\beta_1, \ldots,\beta_d)^\top\), we conclude likewise owing 
to the finiteness of the Laplace transform of $\beta_0 +\beta^\top V$, which holds true thanks to condition~\eqref{eq:assumption_novikov} (since  Assumption~\ref{assm:gen} is in force) with constant \(a(\frac12) =2\).
\end{Remark}

\noindent As a second preliminary, we recall the following result on Volterra convolution equations and the associated resolvent.
\begin{Lemma}[Volterra integral equation]
\label{lem:volterra_resolvent}
Let $K:\mathbb{R}_+\rightarrow\mathbb{R}^{d\times d}$ be a locally integrable kernel, i.e., $K\in L^1_{\mathrm{loc}}(\mathbb{R}_+;\mathbb{R}^{d\times d})$.
 For a given function $g\in L^1_{\mathrm{loc}}(\mathbb{R}_+;\mathbb{R}^{d})$ and $D\in\mathbb{R}^{d\times d}$ the Volterra convolution equation
\begin{equation}
x(t)=g(t)+\int_0^t K(t-s)Dx(s)\,ds.
\label{eq:volterra_conv}
\end{equation}
(also reading $x= g + K*Dx$)  has a unique solution $x\in L^1_{\mathrm{loc}}(\mathbb{R}_+;\mathbb{R}^{d}),$
given by
\begin{equation}\label{eq:solutionresol}
    \forall\, t\ge 0, \quad x(t) = g(t) - \int_0^t R_{D}(t-s)g(s)ds\; \quad \text{or equivalently,} \quad x= g- R_{D}*g,
\end{equation}
         $R_D\in L^1_{\mathrm{loc}}(\mathbb{R}_+;\mathbb{R}^{d\times d})$ is the resolvent
of the second kind associated with the kernel $-KD$, characterized by the relation ( written in terms of convolution)
\begin{equation}
R_D*KD=KD*R_D =KD+R_D.
\label{eq:resolvent_identity}
\end{equation}
\end{Lemma}
\noindent \emph{Proof:} Note that $-KD\in L^1_{\mathrm{loc}}(\mathbb{R}_+;\mathbb{R}^{d\times d})$ guarantees by {\cite[Chapter 2, Theorems 3.1]{gripenberg1990}} that its second kind resolvent function $R_D$ defined in~\eqref{eq:resolvent_identity} always exists.
The existence of a unique solution $x\in L^1_{\mathrm{loc}}(\mathbb{R}_+;\mathbb{R}^{d})$ in~\eqref{eq:solutionresol} follows from {\cite[Chapter 2, Theorems 3.5]{gripenberg1990}}. \hfill $\Box$

\medskip
\noindent {\bf Remarks:} 1. For every $D \!\in \R^{d\times d}$,  the \textit{ resolvent or Solvent core} of the second kind $R_{D}$ associated to the matrix-valued kernel $-KD$, is defined as the unique solution -- if it exists --  to the deterministic matrix-valued Volterra equation~\eqref{eq:resolvent_identity}. Its solution admits the formal \textit{Neumann series expansion}
\begin{equation}\label{eq:NeumExp}
	R_{D} =  - \sum_{k\ge 1} \big( K D\big)^{\star k}.
\end{equation}
where \(\big( K D\big)^{\star k}\) denotes the $k$-fold $\star$ product of $KD$ with itself  with the convention, $\big( K D\big)^{\star 1}= KD$.

2. If $K$ is continuous on $(0,\infty)$, then so is $R_D$. Moreover, if $K\equiv I_d$, then $R_D(t)=e^{Dt}$ and one checks that $R_D\in L^1(\mathbb{R}_+)$
if and only if all eigenvalues of $D$ have strictly negative real parts.

\subsection{Proofs of Proposition~\ref{prop:ExpoExp_riccati} and Theorem~\ref{Thm:ExpoUtilityGeneral}}\label{sect:proofMresult2}
    \noindent {\bf Proof of Proposition~\ref{prop:ExpoExp_riccati}: } We set $Y_t = \sum_{i=1}^d \int_0^t \zeta_i V^i_s ds + \int_t^T \Big(\int_{E}h_{\gamma}\big(\psi(T-s)^\top\eta(e))\nu_0(de)+\sum_{i=1}^d \big(\zeta_i -\frac{ \theta_i^2}{2\gamma}  +  F_i(s,\psi(T-s))\big)  g^i_t(s) \Big)ds$.
	The dynamics of \( Y \) can readily be obtained by recalling \( g_t(s) \) from ~\eqref{eq:processg} and by observing that for fixed \( s \), the dynamics of \( t \to g_t(s) \) are given by \(	dg_t(s) = K(s-t) \, dZ_t \quad t \leq s.\)
	Since \( g_t(t) = V_t \), invoking the stochastic Fubini theorem in \cite[Theorem 2.2]{Veraar2012} (see also, ~\cite[Theorem 65, Chapter IV]{Protter2005}, show that the dynamic of $Y$ reads as: 
	\begin{align*}
		&\, dY_t = \Big(\sum_{i=1}^d \zeta_i V^i_t  - \big(\zeta_i-\frac{ \theta_i^2}{2\gamma}  + F_i(s,\psi(T-s))\big) V^i_t - \int_{E}h_{\gamma}\big(\psi(T-t)^\top\eta(e)\big)\nu_0(de) \Big)dt\\
        &\hspace{1.5cm}+\; \sum_{i=1}^d \int_t^T  {K_i}(s-t) \big(\zeta_i-\frac{ \theta_i^2}{2\gamma} + F_i(s,\psi(T-s))\big) ds dZ^i_t\\
        &= \Big( - \sum_{i=1}^d \Big(-\frac{ \theta_i^2}{2\gamma}  -\theta_i \rho_i \sigma^v_i \varsigma^i(s) \psi^i + \gamma \frac {(\sigma^v_i)^2} 2 (1-\rho^2_i) (\varsigma^i(s)\psi^i)^2 + \int_{E}h_{\gamma}\big(\psi(T-t)^\top\eta(e)\big)\nu_i(de) \Big) V^i_t \\&\hspace{0.5cm}- \int_{E}h_{\gamma}\big(\psi(T-t)^\top\eta(e)\big)\nu_0(de) \Big)dt + \;  \sum_{i=1}^d \psi^i(T-t)\sigma^v_i \varsigma^i(t) \sqrt{V^i_t}dW^i_t +  \sum_{i=1}^d \psi^i(T-t) \int_{E}\eta^i(e)\tilde{N}(dt,de) \\ 
        &= \; \Big(\sum_{i=1}^d \Big(\frac{1}{2\gamma}\left( \theta_i^2 + \gamma\rho_i \sigma^v_i \varsigma^i(t) \psi^i(T-t )\right)^2 - \gamma \frac {(\sigma^v_i)^2} 2 (\varsigma^i(t)\psi^i(T-t))^2  \Big) V^i_t - \int_{E}h_{\gamma}\big(\psi(T-t)^\top\eta(e)\big)\xi(V_t, de) \Big)dt \\
        &\qquad\qquad+   \sum_{i=1}^d \psi^i(T-t)\sigma^v_i \varsigma^i(t) \sqrt{V^i_t}dW^i_t +   \sum_{i=1}^d \psi^i(T-t) \int_{E}\eta^i(e)\tilde{N}(dt,de)
	\end{align*}
	where we changed variables and used the inhomogeneous Riccati--Volterra equation \eqref{eq:RiccatiExpTilpsi1}-\eqref{eq:RiccatiExpTilpsi2} for $\psi$ for the second equality.  Finally, observing that 
	\begin{equation*}
		\left| \lambda_t + \gamma \Sigma \Lambda_t \right|^2  = \;  \sum_{i=1}^d \left( \theta_i^2 + \gamma \rho_i \sigma^v_i \varsigma^i(t) \psi^i(T-t )\right)^2 V^i_t \quad \text{and} \quad \left| \Lambda_t\right|^2  = \; \sum_{i=1}^d (\sigma^v_i)^2 (\varsigma^i(t)\psi^i(T-t))^2  V^i_t
	\end{equation*}
    together with \(Y_T=\zeta\), we get that $(Y, \Lambda, U)$ as defined in ~\eqref{eq:SolbsdeExpo} solves the BSDEJ~\eqref{eq:bsdeExpo}.
{It remains to show that $\left(Y, \Lambda, U\right) \in \mathbb{S}^{p}_{\F}([0,T],\R) \times L^2_{\F}([0,T], \R^d) \times L^2_{\F}([0, T],\tilde{N})$ for some sufficiently large \(p>1\). 
\noindent From the claim~$2.$ of the Remark following Theorem~\ref{Thm:ExpoUtilityGeneral} ($\exp(\gamma Y)$ is bounded), we have that  \( Y_t \leq 0\) for all $ t \in [0, T]$, $\p$-$\as$
Moreover, we have the representation, \(Y_t:=\E \Big[ Y_T - \int_t^T  f\big(Y_s, \Lambda_s, U_s\big) ds \mid \mathcal F_t\Big]\) where \(Y_T=\zeta=:\int_0^T\sum_{i=1}^d \zeta_i V^i_s ds\) and \(f\) in~\eqref{eq:GammaDef2} is defined in~\eqref{eq:bsdeExpo}. Consequently, we have:
\[\forall\, t\in[0,T], \quad Y_t\leq \E \Big[ Y_T +\frac{1}{2\gamma}\int_0^T \left| \lambda_s + \gamma \Sigma \Lambda_s \right|^2 ds \mid \mathcal F_t\Big]:=M_t\]
Using the elementary inequality \((a+b)^p \le 2^{(p-1)^+}\big(a^p + b^p\big)\) for \(a,b>0\), one gets
\begin{equation}\label{eq:boundOpt}
	|\lambda_s +  \gamma \Sigma\Lambda_s|^2 \leq 2(|\lambda_s|^2 +  \gamma ^2|\Sigma\Lambda_s|^2) \leq 2(1 + \gamma ^2|\Sigma|^2)(|\lambda_s|^2 +  |\Lambda_s|^2) \leq 2(1 + \gamma ^2)(1 + |\Sigma|^2)(|\lambda_s|^2 +  |\Lambda_s|^2).
\end{equation}
Consequently, by simple calculation, using the the inequality $|z|^q\leq c_{q}e^{|z|}, \, \forall z \in \R$, together with  Cauchy-Schwarz inequality in the third line:
\begin{align}
	&\,\E \Big[ Y_T + \frac{1}{2\gamma}\int_0^T \left| \lambda_s+ \gamma\Sigma \Lambda_s \right|^2ds \Big] \leq c_{1} \E \big[ \exp\big(|M_T|\big)\big]\label{eq:boundExpoM_T}\\
    &\hspace{1.2cm}\leq c_{1}\E \Big[ \exp\Big(\int_0^T\sum_{i=1}^d |\zeta_i|V^i_s ds+ \frac{(1 + \gamma ^2)}{\gamma}(1 + |\Sigma|^2)\int_0^T \big(|\lambda_s|^2 +  |\Lambda_s|^2\big)ds\Big) \Big]\nonumber\\
    &\hspace{1.2cm}\leq c_{1} \Big(\E \Big[ \exp\Big(2\int_0^T\sum_{i=1}^d |\zeta_i| V^i_s ds\Big) \Big]\Big)^\frac12\Big(\E \Big[ \exp\Big(2\frac{(1 + \gamma ^2)}{\gamma}(1 + |\Sigma|^2)\int_0^T \big(|\lambda_s|^2 +  |\Lambda_s|^2\big)ds\Big) \Big]\Big)^\frac12<\infty\nonumber
\end{align}
thanks to the Novikov-type condition~\eqref{eq:assumption_novikov} with constant \(a(\frac{2(1 + \gamma ^2)}{\gamma}) = \frac{2(1 + \gamma ^2)}{\gamma}(1 + |\Sigma|^2)\) , together with the exponential affine representation~\eqref{eq:laplace} in Theorem~\ref{Thm:VolSqrtAll}.  Therefore, $M_t$ is a martingale under $\P$.  
Let $ p > 1 $.
By {virtue of Doob's maximal inequality} together with the inequality $|z|^q\leq c_{q}e^{|z|}, \, \forall z \in \R$: 
\begin{align*}
	&\, \E \Big[ \sup_{ t \in [0, T]} \big| Y_t \big|^p \Big] \leq \E \Big[ \sup_{ t \in [0, T]} \big| M_t \big|^{p}\Big] \leq \Big(\frac{p}{p-1}\Big)^{p} \E \Big[ |M_T|^p\Big]\leq c_{p} \Big(\frac{p}{p-1}\Big)^{p} \E \big[ \exp\big(|M_T|\big)\big]< \infty
\end{align*}
which is finite owing to ~\eqref{eq:boundExpoM_T}.  Therefore, $\E\big[ \sup_{ t \in [0, T]} |Y_t|^p \big] < \infty $ holds. 

\noindent As for $\Lambda$, it is clear that it belongs to $L^2_{\F}([0,T], \R^d)$ since $\varsigma$ and $\psi$ are bounded, $\psi$ is continuous thus bounded and $\E \Big[\int_0^T  \sum_{i=1}^d V^i_s ds \Big] <  \infty$ by \eqref{eq:moments V1}. \\
Moreover $U \in L^2_{\F}([0, T],\tilde{N})$ as for $\Lambda$ due the the boundary property of $\psi$ and integrability condition in \eqref{assump:kernelJumpVolterra}~$(ii)$.
This complete the Proof.\hfill $\Box$

    \medskip
    \noindent{\bf Proof of Theorem~\ref{Thm:ExpoUtilityGeneral}}
	\noindent {\sc Step~1} (\textit{Proof of the Martingale optimality principle for \(J^\alpha\).})
 	We show that $J_t^{\alpha}$ in~\eqref{eq:ansatzExpo} fulfills the martingale optimality principle. 
	For the first condition, note that $Y_T=\zeta$ and hence $J_T^{\alpha}=-\frac{1}{\gamma}\exp\left(-\gamma  (X^\alpha_T-\zeta)\right) = U(X^\alpha_T)$. Since $Y_0$ is a constant independent of $ \alpha \in \cA$, $J^\alpha_0 = -\frac{1}{\gamma}\exp\left(-\gamma  e^{\int_{0}^{T} r(s)ds}x_0\right) \exp( \gamma Y_0)$ is a constant independent of $ \alpha \in \cA$ and thus the second condition is also satisfied.
	In order to show that the third condition is also fulfilled, we apply Itô's formula on $J_t^{\alpha}$ defined in~\eqref{eq:ansatzExpo}. 
    Let's us consider the Problem~\eqref{Expobj} with an arbitrary strategy $\alpha \in \mathcal{A}$.
	To ease notations, we set \(L_t :=-\gamma  e^{\int_{t}^{T} r(s)ds}X^\alpha_t+ \gamma Y_t\) and $h_t := \lambda_t + \gamma\Sigma \Lambda_t $. Note that \(J^\alpha_t := -\frac{1}{\gamma}\exp(L_t)\) 
    ,we write by It\^o's lemma:
	\begin{align*}
		dL_t &= \Big(\gamma r(s) e^{\int_{t}^{T} r(s)ds}X^\alpha_t - \gamma e^{\int_{t}^{T} r(s)ds}\big( r(t) X^{\alpha}_t  + \alpha_t^\T \lambda_t \big) \Big) dt - \gamma e^{\int_{t}^{T} r(s)ds} \alpha_t^\T dB_t  + \gamma dY_t = \gamma \Lambda_t^\top dW_t \\
		& - \gamma e^{\int_{t}^{T} r(s)ds} \alpha_t^\T dB_t +  \gamma \int_{E}U_t(e)N(dt,de) - \gamma \big(  e^{\int_{t}^{T} r(s)ds} \alpha_t^\T \lambda_t + f(Y_t,\Lambda_t, U_t)+\int_{E}U_t(e)\xi(V_t, de) \big) dt
	\end{align*}
	where \(f\) is the driver in Equation~\eqref{eq:GammaDef2}. Consequently, by It\^o-L\'evy's product rule, combined with the property of \(Y\) in Equation~\eqref{eq:bsdeExpo} of Proposition~\ref{prop:ExpoExp_riccati}, one may write
	\begin{align*}
		dJ_t^{\c} &= \; J_{t_-}^{\c} \big( - \gamma e^{\int_{t}^{T} r(s)ds} \alpha_t^\T \lambda_t - \gamma f(Y_t,\Lambda_t, U_t)- \gamma \int_{E}U_t(e)\xi(V_t, de) + \frac{\gamma^2}{2} e^{\int_{t}^{T} 2r(s)ds} \c^\T_t \c_t \\
        &-\gamma^2 e^{\int_{t}^{T} r(s)ds} \c^\T_t \big( \Sigma \Lambda_t\big) + \frac{\gamma^2}{2} \left| \Lambda_t \right|^2 \big) dt + J_{t_-}^{\c} \big(- \gamma e^{\int_{t}^{T} r(s)ds} \c_t^\T dB_t +  \gamma \Lambda_t^\top dW_t +  \int_{E}\big(e^{\gamma U_t(e)}-1\big)N(dt,de)\big) \\
        &= J_{t_-}^{\c} \Big( D_t(\alpha_t) dt - \gamma  e^{\int_{t}^{T} r(s)ds} \c_t^\T dB_t + \gamma \Lambda_t^\top dW_t +  \int_{E}\big(e^{\gamma U_t(e)}-1\big)\widetilde{N}(dt,de) \Big).
	\end{align*}
	where the drift factor takes the form: 
	\begin{align*}
		D_t(\alpha) = & \frac{\gamma^2}{2} e^{\int_{t}^{T} 2r(s)ds}  \c^\T \c -\gamma  e^{\int_{t}^{T} r(s)ds}  \c^\T h_t  + \frac{1}{2} h_t^\T h_t.
	\end{align*}
	Differentiating \(D_t(\alpha) \) with respect to \(\alpha\) and checking the second order condition, the maximizer \(\alpha^*_t:=\frac1\gamma  e^{-\int_{t}^{T} r(s)ds} h_t =\frac1\gamma  e^{-\int_{t}^{T} r(s)ds} \left( \lambda_t +  \gamma \Sigma \Lambda_t \right)\) for every \(t\in[0,T]\) that is the strategy given by Equation~\eqref{Eq:alpha_GeneralExpo*}.
	Evaluating the drift factor \(D_t \) at \(\alpha^*_t\) show that \(D_t(\alpha^*_t)\) vanishes to \(0\). 
	Note $D_t(\alpha)$ is a quadratic function on $\alpha$ and $\gamma - 1 < 0$. As $D_t(\alpha^*_t) =0$, then $D_t(\alpha_t) \geq 0$ for any admissible strategy \(\alpha\).
	Moreover, solving the stochastic differential equation for $J_t^{\alpha}$ yields
	\begin{equation}\label{eq:JExpo}
	    J_t^{\alpha}= -\frac{\exp(Y_0)}{\gamma}\exp\left(-\gamma  e^{\int_{0}^{T} r(s)ds}x_0\right)e^{\int^t_0 D_s(\alpha_s) ds} F_t^\alpha
	\end{equation}
		where
	\begin{align*}
       F_t^\alpha&= \mathcal E\Big(\int_0^t -\gamma  e^{\int_{s}^{T} r(u)du} \c_s^\T dB_s + \gamma \Lambda_s^\top dW_s +  \int_0^t\int_{E}\big(e^{\gamma U_s(e)}-1\big)\widetilde{N}(ds,de) \Big) 
	\end{align*}
	Following our assumptions on the admissible strategies~\ref{eq:AdmStrat2} and Proposition~\ref{prop:ExpoExp_riccati}, we have that
	$\left(\alpha, \Lambda, U\right) \in {L^{2,loc}_{\F}([0,T], \R^d)^2 \times L^2_\F([0,T],\tilde{N})}$ and thus, the stochastic exponential $F_t^\alpha$ is a $(\mathbb{F},\mathbb{P})$-local martingale (which follows from the basic properties of the Dool\'{e}ans-Dade exponential). Therefore, there exists a sequence of stopping times $\{\tau_n\}_{n\geq1}$ satisfying $\lim_{n \rightarrow \infty} \tau_n = T$, $\p$-$\as$, such that that $F^\alpha_{t \wedge \tau_n}$ is a positive martingale for every $n$.
	
	\noindent Furthermore, $- \frac{\exp(Y_0)}{\gamma} \exp\big[ - \gamma e^{ \int^T_0 r(u)du} x_0 \big] e^{\int^t_0 D_s(\alpha_s) ds}$ is non-increasing. Therefore, $J^\alpha_{t\wedge \tau_n}$ is a supermartingale. Then for $s \leq t$, $\E[ J^\alpha_{t\wedge \tau_n} | \cF_s] \leq J^\alpha_{s\wedge \tau_n}$. It implies that for any set $ A \in \cF_s$, 
	\begin{equation*}
		\E[ J^\alpha_{t\wedge \tau_n} \id_A] \leq \E[J^\alpha_{s\wedge \tau_n}\id_A], \quad s \leq t.
	\end{equation*}
	for every $n$.  Since $\left(\exp\big[ - \gamma e^{ \int^T_{t\wedge \tau_n} r(u) du} X_{t \wedge \tau_n} \big]\right)_n$ is uniformly integrable (~\ref{eq:AdmStrat2}) and $\exp(\gamma Y)$ is bounded (see claim~$2.$ of the Remark following Theorem~\ref{Thm:ExpoUtilityGeneral}), $\left(J^\alpha_{t\wedge \tau_n} \right)_n$ and $\left(J^\alpha_{s\wedge \tau_n} \right)_n$ are uniformly integrable. Let $n \rightarrow \infty$, then $\E[ J^\alpha_t \id_A] \leq \E[J^\alpha_s \id_A]$. Then we deduce that $J^\alpha$ is a supermartingale  for every arbitrary admissible strategy $\alpha$.
	It remains to show that $J_t^{\alpha^*}$ is a true martingale for the optimal strategy $\alpha^*$ in which case $e^{\int_0^t D_s(\alpha_s^*)ds}=1$ and therefore the above equation~\eqref{eq:JExpo}, when  inserting the candidate~\eqref{Eq:alpha_GeneralExpo*} for the optimal strategy \(\alpha^*\) reads $J_t^{\alpha^*}=-\frac{\exp(Y_0)}{\gamma}\exp\left(-\gamma  e^{\int_{0}^{T} r(s)ds}x_0\right) F_t^{\alpha^*}$. For $\alpha_t = \alpha^*_t$, $F_t^{\alpha^*}$ is a martingale by Lemma~\ref{lm: extended_m_EG_lemma}. 
	Subsequently, $J^{\alpha^*}_t$ is a true martingale and taking the expectation, its boils down that:
	\begin{align*}
	\E_{x_0,V_0}{\big[-\frac{1}{\gamma}\exp\big(-\gamma (X^{\alpha^*}_T-Y_T)\big) \big]}&:=\E_{x_0,V_0}{\big[J_T^{\alpha^*}\big]}=-\frac{\exp(Y_0)}{\gamma}\exp\Big(-\gamma  e^{\int_{0}^{T} r(s)ds}x_0\Big) \E_{x_0,V_0}{\big[F_T^{\alpha^*} \big]}\\
    &= -\frac{\exp(Y_0)}{\gamma}\exp\Big(-\gamma  e^{\int_{0}^{T} r(s)ds}x_0\Big).
    \end{align*}
	where the last inequality comes from the fact that \(F_T^{\alpha^*}\) is a true martingale. This completes the  proof of the equality~\eqref{eq:valExpo}.
    
	\smallskip
	\noindent 
	  We have verified all conditions required by martingale optimality principle in Definition~\ref{def:Martopt}, except for the admissibility of $\alpha^*$. It becomes straightforward by Step~2 below, that $\alpha^*$ in equation~\eqref{Eq:alpha_GeneralExpo*} is admissible.

    \smallskip
	\noindent {\sc Step~2} (\textit{Proof of Admissibility:})
	The admissibility of the candidate in equation~\eqref{Eq:alpha_GeneralExpo*} for the optimal portfolio strategy $\alpha^*$, and, on the way, the bound in equation~\eqref{eq:boundExpo} are  established in \cite[Proof of Theorem~3.12]{Gnabeyeu2026b}. 
     The proof is complete. \hfill $\Box$

    \subsection{Proof of Proposition~\ref{prop:ExpoPower_riccati} and Theorem~\ref{Thm:powerUtilityGeneral}}\label{sect:proofMresult1}
\noindent {\bf Proof of Proposition~\ref{prop:ExpoPower_riccati}: } We set $Y_t =  \gamma\int_t^T r(s) ds + \int_t^T\int_{E}h\big(\psi(T-s)^\top\eta(e)\big)\nu_0(de)ds \sum_{i=1}^d\int_t^T  \big(\frac{\gamma \theta_i^2}{2(1-\gamma)}  +  F_i(s,\psi(T-s))\big) g^i_t(s) ds, \quad t \leq T.$
	The dynamic of \( Y \) can readily be obtained by recalling \( g_t(s) \) from ~\eqref{eq:processg} and by observing that for fixed \( s \), the dynamics of \( t \to g_t(s) \) are given by \(	dg_t(s) = K(s-t) \, dZ_t \quad t \leq s.\)
	Since \( g_t(t) = V_t \), it follows by stochastic Fubini's theorem, see \cite[Theorem 2.2]{Veraar2012}, that the dynamics of $Y$ reads as  
	\begin{align*}
		&\,dY_t 
        =\Big(-\gamma r(t)  - \int_{E}h\big(\psi(T-t)^\top\eta(e)\big)\nu_0(de) \Big)dt - \sum_{i=1}^d \Big(\big(\frac{\gamma \theta_i^2}{2(1-\gamma)}  + F_i(s,\psi(T-s))\big) V^i_t dt\\
        &\hspace{1cm}+\; \int_t^T  {K_i}(s-t) \big(\frac{\gamma \theta_i^2}{2(1-\gamma)}  + F_i(s,\psi(T-s))\big) ds dZ^i_t\Big) = -\Big( \gamma r(t) + \sum_{i=1}^d \Big(\frac{\gamma \theta_i^2}{2(1-\gamma)}  \\
        &\hspace{.3cm}+  \frac{\gamma}{1-\gamma} \theta_i \rho_i \sigma^v_i \varsigma^i(t) \psi^i(T-t) + \frac {(\sigma^v_i)^2} 2 \big(  1 +\frac{\gamma}{1-\gamma}\rho_i^2\big) (\varsigma^i(t)\psi^i(T-t))^2 + \int_{E}h\big(\psi(T-t)^\top\eta(e)\big)\nu_i(de)\Big) V^i_t \\
		&\hspace{.5cm}+ \int_{E}h\big(\psi(T-t)^\top\eta(e)\big)\nu_0(de)  \Big)dt + \;  \sum_{i=1}^d \psi^i(T-t)\sigma^v_i \varsigma^i(t) \sqrt{V^i_t}dW^i_t + \sum_{i=1}^d \psi^i(T-t) \int_{E}\eta^i(e)\tilde{N}(dt,de)\\ 
        &\hspace{.5cm}= \Big(-\gamma r(t)-  \sum_{i=1}^d \Big(\frac{\gamma}{2(1-\gamma)}\left( \theta_i^2 + \rho_i \sigma^v_i \varsigma^i(t) \psi^i(T-t )\right)^2 + \frac {(\sigma^v_i)^2} 2 (\varsigma^i(t)\psi^i(T-t))^2   \Big) V^i_t\\
        &\hspace{.3cm}-\;\int_{E}h\big(\psi(T-t)^\top\eta(e)\big)\xi(V_t, de) \Big)dt + \sum_{i=1}^d \psi^i(T-t)\sigma^v_i \varsigma^i(t) \sqrt{V^i_t}dW^i_t +   \sum_{i=1}^d \psi^i(T-t) \int_{E}\eta^i(e)\tilde{N}(dt,de) 
	\end{align*}
	where we changed variables and used the inhomogeneous Riccati--Volterra equation \eqref{eq:RiccatiPower2} for $\psi$ for the second equality.  Finally, observing that 
	\begin{equation*}
		\left| \lambda_t + \Sigma \Lambda_t \right|^2  = \;  \sum_{i=1}^d \left( \theta_i^2 + \rho_i \sigma^v_i \varsigma^i(t) \psi^i(T-t )\right)^2 V^i_t \quad \text{and} \quad \left| \Lambda_t\right|^2  = \; \sum_{i=1}^d (\sigma^v_i)^2 (\varsigma^i(t)\psi^i(T-t))^2  V^i_t
	\end{equation*}
    together with \(Y_T=0\), we get that $(Y, \Lambda, U)$ as defined in ~\eqref{eq:SolbsdePower} solves the BSDEJ~\eqref{eq:bsdePower}.
{It remains to show that $\left(Y, \Lambda, U\right) \in \mathbb{S}^{\infty}_{\F}([0,T],\R) \times L^2_{\F}([0,T], \R^d) \times L^2_{\F}([0, T],\tilde{N})$. 

\noindent One check with It\^o-L\'evy's Lemma together with martingale representation that (see also~\cite[proof of Proposition 3.6]{Gnabeyeu2026b}) for \( \; 0\leq t\leq T\):
		\begin{equation}\label{eq:Power_ito_gamma}
			\exp\big(Y_t\big)=:\E^{\mathbb{P}} \Big[ \exp\Big(\int_t^T  \big(\gamma r(s) + \frac{\gamma}{2(1-\gamma)} \left| \lambda_s+ \Sigma \Lambda_s \right|^2\big) ds\Big) \mid \mathcal F_t\Big].
		\end{equation}
	\noindent which ensures that $\exp\big(Y_t\big)>0$ $\P-a.s.$, since $V_t$ is non-negative ($V\in \mathbb R^{d}_+$), $r(t) > 0$ is deterministic, and $ 1-\gamma \leq 1$. We then have in view of~\eqref{eq:Power_ito_gamma} that there exists some positive constant \(1>m>0\) such that $\exp\big(Y_t\big) \geq m > 0$ for every \(t\in [0,T]\). Moreover, we have the representation, \(Y_t:=\E \Big[ Y_T - \int_t^T  f\big(Y_s, \Lambda_s, U_s\big) ds \mid \mathcal F_t\Big]\) where \(Y_T=0\) and \(f\) in~\eqref{eq:GammaDefPower} is defined in~\eqref{eq:bsdePower}, as a non negative-function. It boilds down that \(\ln(m)\leq Y_t \leq 0\) for all $ t \in [0, T]$, $\p$-$\as$ and \(Y\) is  essentially bounded. 
    
\noindent As for $\Lambda$, it is clear that it belongs to $L^2_{\F}([0,T], \R^d)$ since $\varsigma$ and $\psi$ are bounded, $\psi$ is continuous thus bounded and $\E \Big[\int_0^T  \sum_{i=1}^d V^i_s ds \Big] <  \infty$ by \eqref{eq:moments V1}. \\
Moreover $U \in L^2_{\F}([0, T],\tilde{N})$ follows from the boundedness property of $\psi$ together with the integrability condition $(ii)$ in \eqref{assump:kernelJumpVolterra}.
This complete the Proof\hfill $\Box$

    \medskip
     \noindent {\bf Proof of Theorem~\ref{Thm:powerUtilityGeneral}:}
    \smallskip
	\noindent {\sc Step~1} (\textit{Proof of the Martingale optimality principle for \(J^\alpha\).})
    We show that $J_t^{\alpha}$ fulfills the martingale optimality principle in Definition~\ref{def:Martopt}. 
	For the first condition, note that $Y_T=0$ and hence $J_T^{\alpha}=\frac{1}{\gamma}(X_T^{\alpha})^{\gamma}$. Since $\exp(Y_0)$ is a constant independent of $ \alpha \in \cA$, $J^\alpha_0 = \frac{x^\gamma_0}{\gamma} \exp(Y_0)$ is a constant independent of $ \alpha \in \cA$ and consequently, the second condition is also satisfied. In order to show that the third condition is also fulfilled, we apply Itô's formula on $J_t^{\alpha}$ defined in~\eqref{eq:ansatzPower}. 
    
    	\smallskip
	\noindent 
	Let's us consider the Problem~\eqref{obj_power} with an arbitrary strategy $\alpha \in \mathcal{A}$ from Equation~\eqref{eq:AdmStrat}.
	To ease notations, we set \(L_t :=
	\int_0^t\gamma \big(r(s) + \alpha_s^\top \lambda_s - \tfrac12 \left|\alpha_s\right|^2\big)\,ds
	+
	\int_0^t \gamma\alpha_s^\top \, dB_{s} + Y_t\) and $h_t = \lambda_t + \Sigma \Lambda_t $. 
     Note that \(J^\alpha_t := \frac{x_0^\gamma}{\gamma}\exp(L_t)\) 
    ,we write by It\^o's lemma:
	\begin{align*}
		dL_t &= \gamma \big(r(t) + \alpha_t^\top \lambda_t - \tfrac12 \left|\alpha_t\right|^2\big) dt + \gamma \alpha_t^\T dB_t  +  dY_t = \Lambda_t^\top dW_t + \gamma \alpha_t^\T dB_t \\
		& +  \int_{E}U_t(e)N(dt,de) + \big( \gamma \big(r(t) + \alpha_t^\top \lambda_t - \tfrac12 \left|\alpha_t\right|^2\big)- f(Y_t,\Lambda_t, U_t)-\int_{E}U_t(e)\xi(V_t, de) \big) dt
	\end{align*}
	where \(f\) is the driver in Equation~\eqref{eq:GammaDefPower}. Consequently, for any admissible strategy $\alpha \in \mathcal A$, using the Riccati BSDEJ~\eqref{eq:bsdePower} in Proposition~\ref{prop:ExpoPower_riccati} together with It\^o-L\'evy's rule applied to~\eqref{eq:ansatzPower} yield: 
	\begin{align*}
		dJ_t^{\c} &= \; J_{t_-}^{\c} \big( \gamma r(t) + \gamma \alpha_t^\T \lambda_t - \tfrac\gamma2 \left|\alpha_t\right|^2- f(Y_t,\Lambda_t, U_t)- \int_{E}U_t(e)\xi(V_t, de) + \frac{\gamma^2}{2} \c^\T_t \c_t \\
        &+\gamma \c^\T_t \big( \Sigma \Lambda_t\big) + \frac{1}{2} \left| \Lambda_t \right|^2 \big) dt + J_{t_-}^{\c} \big(\gamma \c_t^\T dB_t +  \Lambda_t^\top dW_t +  \int_{E}\big(e^{ U_t(e)}-1\big)N(dt,de)\big) \\
        &= J_{t_-}^{\c} \Big( D_t(\alpha_t) dt + \gamma \c_t^\T dB_t + \Lambda_t^\top dW_t +  \int_{E}\big(e^{ U_t(e)}-1\big)\widetilde{N}(dt,de) \Big).
	\end{align*}
	where the drift factor takes the form: 
	\begin{align*}
		D_t(\alpha) = & \frac{\gamma(\gamma-1)}{2}  \alpha^\T \alpha +\gamma  \alpha^\T h_t - \frac{\gamma}{2(1-\gamma)} h_t^\T h_t.
	\end{align*}
    Now, we perform a pointwise maximization of the drift.
    \noindent Differentiating \(D_t(\alpha) \) with respect to \(\alpha\) and checking the second order condition, one obtains the maximizer \(\alpha^*_t=\frac{1}{1 - \gamma}h_t=\frac{1}{1 - \gamma} \left( \lambda_t +  \Sigma \Lambda_t \right)\) for every \(t\in[0,T]\) that is the strategy given by Equation~\eqref{Eq:alpha_Generalpower*}. Evaluating the drift factor \(D_t \) at \(\alpha^*_t\) show that \(D_t(\alpha^*_t)\) vanishes to \(0\). Note $D_t(\alpha)$ is a quadratic function on $\alpha$ and $\gamma - 1 < 0$. As $D_t(\alpha^*_t) =0$, then $D_t(\alpha_t) \leq 0$ for any admissible strategy \(\alpha\).
    Now, solving the stochastic differential equation for \(J^\alpha\) yields, {using $Y_T=0$},
	\begin{equation}\label{eq:PowerUlti}
		J^\alpha_t=\frac{x_0^{\gamma}}{\gamma}\exp(Y_0 )e^{\int^t_0 D_s(\alpha_s) ds} F_t^\alpha.
	\end{equation}
	where \(F_t^\alpha\) is given by
	\begin{equation*}
		F_t^\alpha= \mathcal E\Big(\int_0^t \gamma \c_s^\T dB_s + \Lambda_s^\top dW_s + \int_0^t\int_{E}\big(e^{U_s(e)}-1\big)\widetilde{N}(ds,de) \Big).
	\end{equation*}
    Now, since $D_s(\alpha_s)\leq 0$ for any \(\alpha\in\mathcal{A}\), we have $e^{\int^t_0 D_s(\alpha_s) ds}$ is a non-increasing function.
	 By our assumptions on the admissible strategies~\ref{eq:AdmStrat} and Proposition~\ref{prop:ExpoPower_riccati},
	$\left(\alpha, \Lambda, U\right) \in {L^{2,loc}_{\F}([0,T], \R^d)^2 \times L^2_\F([0,T],\tilde{N})}$ and thus the stochastic exponential $F_t^\alpha$ is a $(\mathbb{F},\mathbb{P})$-local martingale (which follows from the basic properties of the Dool\'{e}ans-Dade exponential). Therefore, there exists a sequence of stopping times $\{\tau_n\}_{n\geq1}$ satisfying $\lim_{n \rightarrow \infty} \tau_n = T$, $\p$-$\as$, such that 
	\begin{equation*}
		\E[J^\alpha_{t\wedge \tau_n} | \cF_s] \leq J^\alpha_{s\wedge\tau_n}, \quad s \leq t \leq T,
	\end{equation*}
	for every $n$.  Moreover, since $J^\alpha_t$ is bounded from below by \(0\) ( $J_t^{\alpha}\geq 0$), applying Fatou's Lemma for $n\rightarrow \infty$, we deduce that $J^\alpha_t$ is a supermartingale for every arbitrary admissible strategy $\alpha$.
	It remains to show that $J_t^{\alpha^*}$ is a true martingale for the optimal strategy $\alpha^*$ in which case $e^{\int_0^t D_s(\alpha_s^*)ds}=1$ and hence the above Equation~\eqref{eq:PowerUlti}, when  inserting the candidate~\eqref{Eq:alpha_Generalpower*} for the optimal strategy \(\alpha^*\)  reads $J_t^{\alpha^*}=\frac{ x_0^{\gamma}}{\gamma} \exp(Y_0 ) F_t^{\alpha^*}$. $F_t^{\alpha^*}$ is a \(\P\)-martingale with expectation \(1 \) by Lemma~\ref{lm: extended_m_EG_lemma}. Consequently, taking the expectation, it is straightforward that:
	\[\E_{x_0,V_0}{\big[\frac{(X^{\alpha^*}_T)^\gamma}{\gamma}  \big]}:=\E_{x_0,V_0}{\big[J_T^{\alpha^*}\big]}=\frac{ x_0^{\gamma}}{\gamma} \exp(Y_0 ) \E_{x_0,V_0}{\big[F_T^{\alpha^*} \big]}=\frac{ x_0^{\gamma}}{\gamma} \exp(Y_0 ).\]
	where the last inequality comes from the martingality of \(F_T^{\alpha^*}\) and the desired result about the first part of the proof (\textit{equality~\eqref{eq:valPower}}) is completed. 
	It remains to show the \textit{admissibility} of the optimal portfolio strategies \(\alpha^*\).
    
    \smallskip
	\noindent {\sc Step~2} (\textit{Proof of Admissibility:})
	The admissibility of the candidate in equation~\eqref{Eq:alpha_Generalpower*} for the optimal portfolio strategy $\alpha^*$ as well as the bound in equation~\eqref{eq:boundPower}, are obtained by exactly the same computations as in \cite[Proof of Theorem~3.7]{Gnabeyeu2026b}.
    
   \noindent  We have verified all conditions required by martingale optimality principle. \hfill $\Box$

\subsection{Proof of Theorem~\ref{Thm:logUtilityGeneral} } 
\noindent Note that $J_t^{\alpha}$ satisfies conditions~$1$ and~$2$ of Definition~\ref{def:Martopt} since $Y_T=0$ and $Y_0$ is a constant independent of $ \alpha \in \cA$. In order to show that the third condition is fulfilled, we apply Itô's formula on $J_t^{\alpha}$ which yields
	\(dJ_t^{\alpha}=\big( \alpha_t^\top \lambda_t - \tfrac12 \left|\alpha_t\right|^2- \tfrac12 \left|\lambda_t\right|^2\big)\,dt + \alpha_t^\top \, dB_{t} + \Lambda_t^\top dW_t + \int_{E} U_t(e)\,\tilde{N}(dt,de)\).

\noindent Following our assumptions on the admissible strategies~\ref{eq:AdmStrat} and the spaces to which the solution $(Y, \Lambda, U)$ to the BSDEJ belong, we have that
	$\left(\alpha, \Lambda, U\right) \in {(L^{2}_{\F}([0,T], \R^d))^2 \times L^2_\F([0,T],\tilde{N})}$ and thus, the stochastic integrals in $J^\alpha$ are $(\mathbb{F},\mathbb{P})$ martingales for all $\alpha \in \mathcal{A}$.
    
\noindent  Therefore, $J^\alpha$ is a supermartingale for all $\alpha \in \mathcal{A}$. 
In particular, $J^{\alpha^*}$ is a martingale (the pointwise maximization  of the drift yields \(\c^*_t:=\lambda_t\) for every \(t\in [0,T]\) ). 
The admissibility of $\alpha^*$ follows from Equation~\eqref{eq:assumption_novikov} in the remark~\eqref{rm:ass_} on Assumption~\ref{assm:gen} 
and we conclude that $\alpha^*$ is an optimal strategy for \eqref{obj_log}.

\noindent  The corresponding value function is determined simply by the initial value of the supermartingale \(J_t^\alpha\) at \(\c^*\). Subtituting \(\c^*\) into~\eqref{obj_log} yields:
	\begin{align*}
		\sup_{\alpha(\cdot) \in \mathcal A}\E\Big[\log(X_T^{\alpha})\Big]
		=
		\log(x_0) +
		\E\Big[\int_0^T \big(r(s) + \tfrac12 \left|\lambda_s\right|^2\big)\,ds\Big]= \log(x_0) +
		\int_0^T r(s) \,ds + \tfrac12 \int_0^T(\theta\odot\theta)^\top \E\big[V_s\big]\,ds
	\end{align*}
where $\odot$ denote the Hadamard (pointwise or component-wise) product. Taking the expectation in Equation~\ref{VolSqrt2} reads, owing to regular and stochastic Fubini's theorems 
\begin{equation}\label{eq:ExpectationV} 
		\E\big[V_t\big] = g_0(t) + \int_0^t K(t-s)  D \E\big[V_s\big] ds.
\end{equation}
\noindent Note that, the linear growth in Equation~\eqref{eq:LinearGrowth} together with the moment control in Equation~\eqref{eq:moments V1} enable the unrestricted use of both regular and stochastic Fubini's theorems.
Sufficient conditions for interchanging the order of ordinary integration (with respect to a finite measure) and stochastic integration (with respect to a square integrable martingale) are provided in  \cite[Thm.1]{Kailath_Segall},  and further details can be found in \cite[Thm. IV.65]{Protter2005}, \cite[ Theorem 2.6]{Walsh1986}, \cite[ Theorem 2.6]{Veraar2012}.

\noindent Finally, the claimed value function in~\eqref{eq:valueLog} follows from Lemma~\ref{lem:volterra_resolvent} by noticing that, the unique solution of the linear deterministic Volterra equation~\eqref{eq:ExpectationV} reads:
\begin{equation}\label{eq:ExpectationVSol} 
		\E\big[V_t\big] = g_0(t) - \int_0^t R_D(t-s)  g_0(s) ds.
\end{equation}
where $R_D$ is the resolvent
of the second kind of the matrix-valued kernel $-KD$. (See also, the Wiener-Hopf equation  in~\cite[Proposition 2.4]{EGnabeyeu2025} in the one-dimensional setting).
\noindent This completes the proof and we are done. \hfill $\square$

\medskip\noindent {\bf Provisional remark:} 
As a concluding remark, establishing
     existence and uniqueness results for the time-dependent, multivariate Riccati-Volterra integral equations with L\'evy exponential jump compensator in
\eqref{eq:RiccatiExpTilpsi1}--\eqref{eq:RiccatiExpTilpsi2} and \eqref{eq:RiccatiPower1}--\eqref{eq:RiccatiPower2}, which are used to characterise the optimal strategy,
     constitutes an interesting direction for future research and lies beyond the scope of the present paper. We refer to~\cite{BondiLivieriPulido2024} for results in the specific one-dimensional setting.
     
\medskip
\medskip
\medskip
	\noindent {\bf Acknowledgements:} 
	 The authors would like to thank L. Abbas-Turki, I. Kharroubi, G.  Pag\`es and M. Rosenbaum 
    for their invaluable insights and for many fruitful and inspiring discussions.
\bibliographystyle{alpha}
\bibliography{references}




\appendix
\section{Measure-extended conditional Laplace functional for Affine Volterra Processes with jumps}
	\label{sect-CharacteristicFunction}
	\noindent In this section, we establish the representation result for the conditional Laplace functional of the time-inhomogeneous affine Volterra equation~\eqref{VolSqrt_}-- ~\eqref{VolSqrt2} and prove that it is exponential-affine in the past path. This is a measure-extended generalisation of the result in~\cite{BondiLivieriPulido2024}.
	More generally, we consider the time-inhomogeneous affine Volterra equation~\eqref{VolSqrt_}-- ~\eqref{VolSqrt2} where we assume more generally that the matrix $D$ in the drift is not necessarily a diagonal matrix, but defined as a  matrix-valued function i.e. 
	$D : \mathbb{R}_+ \to \mathcal{M}_d(\mathbb{R})$.
	We assume moreover that such resulting equation~\eqref{VolSqrt_}-- ~\eqref{VolSqrt2}
	has (at least) one \textit{non-negative} weak solution \(V = (V_t)_{t \geq 0}\) defined on some stochastic basis \((\Omega, \mathcal{F}, (\mathcal{F}_t)_{t \geq 0}, \mathbb{P})\).\\
	\noindent To state the main formula in a synthetic form, let us define and then consider for a measure $m \in \mathcal{M}$, the following measure-extended Riccati--Volterra equation:
	\begin{equation}\label{eq:measureFLplce}
		\begin{aligned}
			\forall \, m \in \mathcal{M},\quad \psi(t) &= \int_{[0,t)} K(t-s)\,m(-\dd s) + \int_0^t K(t-s)\,F(T-s,\psi(s))\,\dd s, \quad 0\leq t \leq T\\
			F_i(s, \psi) &= (D^\top(s) \psi)_i + \frac{(\sigma^v_i)^2}{2} (\varsigma^i(s)\psi^i)^2 + \int_{E} h(\psi^\top \eta(e))\nu_i(de) , i=1,\ldots,d; \quad (s,\psi)\in \R_+\times \R.
		\end{aligned}
	\end{equation}
	\newpage
	where $D : \mathbb{R}_+ \to \mathcal{M}_d(\mathbb{R})$, and \( \varsigma : \mathbb{R}_+ \to \mathbb{R} \) are  given continuous function.

	\medskip
	\noindent {\bf Remark:}
	Equation~\eqref{eq:measureFLplce} is written in a forward form. An equivalent expression in backward form is:
	\begin{equation}\label{eq:measureFLplce_}
		\psi(T - t) = \int_t^T K(s-t) \, m(\dd s) + \int_t^T K(s-t) F(s, \psi(T - s))  \, \dd s.
	\end{equation}
	This 
	formulation~\eqref{eq:measureFLplce_} is essential in problems where the system's behavior is determined by a known final state, allowing for the determination of the system's evolution by integrating backwards in time.

	\medskip
	\noindent 
	First note that, for any \(T>0\), from the definition of the convolution of a measure \(m \in \mathcal{M}\) and a function \(f : (0, T] \to \mathbb{R}^d\) in equation~\eqref{eq:convolmeasure}, it is straightforward to check that for each \(p \in [1, \infty]\),
	\(\|f * m\|_{L^p([0, T])} \leq \|f\|_{L^p([0, T])} \, |m|([0, T]).\)

	\medskip
	\noindent In the following theorem, we assume the existence and uniqueness of \( \mathbb{R}^d-\) solution $\psi = \psi(\cdot,m) \in C([0,T], \mathbb{R}^d)$ to the Riccati--Volterra equation~\eqref{eq:measureFLplce}, and we aim to  establish an expression for the conditional Laplace transform of~\eqref{VolSqrt_}--~\eqref{VolSqrt2}
	in terms of the Riccati--Volterra equation~\eqref{eq:measureFLplce}.
	
	\begin{Theorem}\label{Thm:VolSqrtAll} 
		Fix \(T>0\) and suppose that Assumption~\ref{assump:kernelJumpVolterra} holds.
		Consider a measure $m \in \mathcal{M}$
        and assume there exists a solution $\psi = 	\psi(\cdot,m) \in C([0,T], \mathbb{R}^d)$ to the measure-extended Riccati--Volterra equation ~\eqref{eq:measureFLplce}.
		  Then, the following exponential-affine transform formula holds for the measure-extended Laplace transform of \( V \) in ~\eqref{VolSqrt2} for every \(t\in[0,T]\):
		\begin{align}\label{eq:laplace}
			\mathbb{E}\Big[ \exp\Big( \int_t^T V_{s}^\top \, \,m( \mathrm{d}s) \Big) \Big| \mathcal{F}_t \Big]
			&= \exp\Big( \int_t^T g_t(s)^\top \, m( \mathrm{d}s) \nonumber\\
            &\hspace{.2cm}+\;
			\int_t^T \Big( \int_{E}h\big(\psi(T-s)^\top\eta(e)\big)\nu_0(de)+F(s, \psi(T - s))^\top g_t(s)\Big)\, \mathrm{d}s \Big).
		\end{align}
		where the process  $(g_t(s))_{t\leq s}$ for each $s\leq T$, denotes the conditional \(\P-\)expected adjusted process defined in~\eqref{eq:Condprocessg} and given by:
		\begin{align}
			g_t(s) &= g_0(s) + \int_0^t K(s-u) dZ_u, 
			\qquad t \le s, \quad \text{where for } i=1,\ldots,d,\label{eq:AdjDiffForward}\\
			dZ_{t}^i &= (D(t)V_t)_{i}\, dt + \sigma^v_i\varsigma^i(t)\sqrt{V_t^i}\, dW^i_{t} + \int_{E}\eta^i (e)\tilde{N}(dt,de), \quad g_0(t):= \varphi(t) V_0 + \int_0^t K(t-s)\mu(s) ds.\nonumber
		\end{align}
	\end{Theorem}
	\noindent {\bf Proof of Theorem~\ref{Thm:VolSqrtAll}.}	
	\smallskip
	\noindent  Let \( T > 0 \) and consider a measure $m \in \mathcal{M}$ such that
    there exists a unique solution $\psi = 	\psi(\cdot,m) \in C([0,T], \mathbb{R})$ to the measure-extended Riccati--Volterra equation ~\eqref{eq:measureFLplce}.
	Define 
	\[L_t = 	\int_0^{t} V_{s}^\top \,m(ds) + \int_t^T g_t(s)^\top\,m(ds) +  \int_t^T  \Big( \int_{E}h\big(\psi(T-s)^\top\eta(e)\big)\nu_0(de)+F(s, \psi(T - s))^\top g_t(s)\Big)\, \mathrm{d}s \]
	and set \( M = \exp(L) \) so that
    \begin{equation}\label{eq:Mdynamics}
    dM_t = M_{t_-} \Big( dL^c_t + \tfrac{1}{2}\, d\langle L^c \rangle_t +\int_E(\exp(\Delta L_t) - 1) N(dt, de) \Big).
   \end{equation}
   where $L^c$ is the continuous part of $L$ and $\Delta L$ his jump size.
	\noindent Let \( V \) be a solution of  equation~\eqref{VolSqrt_}--~\eqref{VolSqrt2} under Assumption~\ref{assump:kernelJumpVolterra}.   	 
	Then the process \((M_t )_{t\in[0,T]}\) is a
	local martingale on \( [0, T] \), and satisfies by It\^o-Levy's formula \[\frac{\mathrm{d}M_t}{M_{t_-}} = \sum_{i=1}^{d}
	\sigma^v_i \varsigma^i(t)\psi^i(T - t) \sqrt{V_t^i} \, dW_{t}^i + \int_{E}(e^{\psi(T - s)^\top\eta(e)}-1)\tilde{N}(dt,de).\]
    Now, the dynamics of \( L \) can be obtained by recalling \( g_t(s) \) from \eqref{eq:AdjDiffForward} and noting that for fixed \( s \), the dynamics of \( t \mapsto g_t(s) \) are given by 
	\[
	dg_t(s) = K(s-t) \, dZ_t, \;\quad  dZ_t = D(t) V_t dt + \sigma^v\varsigma(t) \sqrt{\diag(V_t)}dW_t + \int_{E}\eta (e)\tilde{N}(dt,de), \quad t \leq s.
	\]
	Since \( g_t(t) = V_t \), it follows by stochastic Fubini's theorem, see \cite[Theorem 2.2]{Veraar2012}, that the dynamics of $G$ reads
	\begin{align*}
		dL_t &=  V_{t}^\top \, m(\mathrm{d}s) -  g_t(t)^\top m(\mathrm{d}s) - F(t, \psi(T - t))^\top g_t(t) \, dt + \int_t^T dg_t(s)^\top \,\Big(m(ds) + F(s, \psi(T - s))\,ds\Big)\\
		&\hspace{1.5cm}-\;  \int_{E}h\big(\psi(T-t)^\top\eta(e)\big)\nu_0(de) dt\\ 
        &= - \Big(F(t, \psi(T - t))^\top V_{t} +  \int_{E}h\big(\psi(T-t)^\top\eta(e)\big)\nu_0(de)\Big)\, dt + \psi(T - t)^\top\,dZ_t
	\end{align*}
	where for the second equality we used the measure-extended Riccati--Volterra equation~\eqref{eq:measureFLplce_}. This implies that \(d\langle L^c \rangle_t =  \sum_{i=1}^{d}
	(\sigma^v_i\varsigma^i(t)\psi^i(T - t))^2 V^i_t\, dt.\)
	Injecting the dynamics of \( dL^c_t \) and \( d\langle L^c\rangle_t \) into that of \(M_t\) in Equation~\eqref{eq:Mdynamics}, we get that
	\begin{align*}
		&\frac{dM_t}{M_{t_-}}
		=
		\sum_{i=1}^{d}
		\big(- F_i(t,\psi(T-t)) + (D^{\top}(t)\psi(T-t))_i 
		+ \frac{(\sigma^v_i)^2}{2}(\varsigma^i(t)\psi^i(T - t))^2\,
		\big) V_{t}^i\, dt -\int_{E} \psi(T - t)^\top\eta(e)\xi(V_t, de) dt \\
		&-\; \int_{E}h\big(\psi(T-t)^\top\eta(e)\big)\nu_0(de) dt+
		\sum_{i=1}^{d}\sigma^v_i
		\varsigma^i(t)\psi^i(T - t)  \sqrt{V_t^i} \, dW_{t}^i + \int_{E}(e^{\psi(T - t)^\top\eta(e)}-1)N(dt,de)\\ 
        &=
		\sum_{i=1}^{d}
		\big(- F_i(t,\psi(T-t)) + (D^{\top}(t)\psi(T-t))_i 
		+ \frac{(\sigma^v_i)^2}{2}(\varsigma^i(t)\psi^i(T - t))^2\,
		+ \int_{E}h\big(\psi(T-t)^\top\eta(e)\big)\nu_i(de)\big) V_{t}^i\, dt  \\
        &+
		\sum_{i=1}^{d}\sigma^v_i
		\varsigma^i(t)\psi^i(T - t)  \sqrt{V_t^i} \, dW_{t}^i + \int_{E}(e^{\psi(T - t)^\top\eta(e)}-1)\tilde{N}(dt,de)= \Lambda_t^\top dW_t +   \int_{E}(e^{U_t(e)}-1)\tilde{N}(dt,de).
	\end{align*}
		where we changed variables in the first equality using
	\[  \sum_{j=1}^d \psi^j(T-t) (D(t) V_t)_j = \sum_{j=1}^d \psi^j(T-t) \sum_{i=1}^d D_{ji}(t) V_t^i = \sum_{i=1}^d V_t^i \sum_{j=1}^d D_{ji}(t) \psi^j(T-t) = \sum_{i=1}^d (D^\top(t) \psi (T-t))_iV_t^i \] 	
	and the drift vanishes in the last equality by definition of $F$ in the Riccati--Volterra equation~\eqref{eq:measureFLplce}. 
	This shows that $M$ is an exponential local martingale of the form
	\begin{equation}\label{eq:MartM}
		M_t = \mathcal{E}\!\left(
		\sum_{i=1}^{d}
		\int_0^t \sigma^v_i\varsigma^i(s)\psi^i(T - s)  \sqrt{V_s^i}\, dW_{s}^i+\int_0^t\int_{E}(e^{\psi(T - s)^\top\eta(e)}-1)\tilde{N}(ds,de)
		\right).
	\end{equation}
	\noindent To obtain \eqref{eq:laplace}, it suffices to prove that \( M \) is a martingale. Indeed, if this is the case then, the martingale property yields using that \(L_T = \int_0^T V_{s}^\top \,m( \mathrm{d}s)\)
		\begin{align*}
			&\ \mathbb{E}\left[ \exp\left( \int_0^T V_{s}^\top \,m( \mathrm{d}s) \right)  \Big| \mathcal{F}_t \right]
			= \mathbb{E}\left[ M_T \Big| \mathcal{F}_t \right] = M_t \\
			&\hspace{4.75cm}= \exp\left(\int_0^{t} V_{s}^\top \,m(ds) + \int_t^T g_t(s)^\top\,m(ds) +  \int_t^T F(s, \psi(T - s))^\top g_t(s)\, \mathrm{d}s \right).
		\end{align*}
	That is, if \( M \) is a true martingale, then the measure-extended Laplace transform of \( V_T \) is given by
		\begin{align}\label{eq:resFourierLaplace}
			\mathbb{E}\left[ \exp\left( \int_t^T  V_{s}^\top \,m( \mathrm{d}s) \right) \Big| \mathcal{F}_t \right]
			= \exp\left( \int_t^T g_t(s)^\top \,m(ds) +  \int_t^T F(s, \psi(T - s))^\top g_t(s)\, \mathrm{d}s \right).
		\end{align}
	\noindent which yields \eqref{eq:laplace}.
	We now argue martingality of \( M \). 
	Since $\psi$ is continuous, it is bounded; likewise, $\varsigma$ is bounded. Therefore the stochastic exponential~\eqref{eq:MartM} is a true $\P$- martingale thanks to Lemma~\ref{lm: extended_m_EG_lemma} with $g_{1} \equiv 0$, $g_{2} \equiv 1  \in L^{\infty}([0,T],\R)$ and $U_s (e):= \psi(T - s)^\top\eta(e)\; \forall e\in E$.
	We conclude that \( M \) is a martingale.
	\noindent This completes the proof and we are done. \hfill $\square$
\end{document}